\documentclass[10pt]{amsart}

\usepackage{amsmath}

\usepackage{breakcites}

\usepackage{amssymb}
\usepackage{dsfont}
  \usepackage{paralist}
  \usepackage{graphics} 
  \usepackage{epsfig} 
\usepackage{graphicx}  \usepackage{epstopdf}
 \usepackage[colorlinks=false]{hyperref}
 \usepackage{comment}
 \usepackage{enumitem}
\usepackage{autonum}
\usepackage{subcaption}
\usepackage{makebox}
\usepackage{booktabs}
\usepackage{todonotes}
\usepackage{appendix}
\usepackage{enumitem}
\usepackage{amsthm}
\usepackage{float}

\hypersetup{colorlinks,linkcolor={blue},citecolor={blue},urlcolor={black}}

\newtheorem{theorem}{Theorem}[section]

\newtheorem{lemma}[theorem]{Lemma}

\newtheorem{remark}{Remark}

\newtheorem{assumption}{Assumption}

\DeclareMathOperator\supp{supp}


\title[The Bayesian approach to inverse Robin problems]{The Bayesian approach to inverse Robin problems}

\bigskip

\begin{document}

\maketitle

\centerline{\scshape Aksel Kaastrup Rasmussen}
{\footnotesize
   \centerline{Department of Applied Mathematics and Computer Science}
 \centerline{Technical University of Denmark}
 \centerline{\texttt{\href{mailto:akara@dtu.dk}{akara@dtu.dk}}}
}

\bigskip
\centerline{\scshape Fanny Seizilles}
{\footnotesize
 \centerline{Centre for Mathematical Sciences}
   \centerline{University of Cambridge}
}
\bigskip

\centerline{\scshape Mark Girolami}
{\footnotesize
 \centerline{Department of Engineering, University of Cambridge}
    \centerline{and The Alan Turing Institute, London}
} 
\bigskip

\centerline{\scshape Ieva Kazlauskaite}
{\footnotesize
 \centerline{Department of Statistical Science, UCL}
    \centerline{and Department of Engineering, University of Cambridge}
    \centerline{\texttt{\href{mailto:i.kazlauskaite@ucl.ac.uk}{i.kazlauskaite@ucl.ac.uk}}}
} 
\bigskip

\begin{abstract}
    In this paper we investigate the Bayesian approach to inverse Robin problems. These are problems for certain elliptic boundary value problems of determining a Robin coefficient on a hidden part of the boundary from Cauchy data on the observable part. Such a nonlinear inverse problem arises naturally in the initialisation of large-scale ice sheet models that are crucial in climate and sea-level predictions. We motivate the Bayesian approach for a prototypical Robin inverse problem by showing that the posterior mean converges in probability to the data-generating ground truth as the number of observations increase. Related to the stability theory for inverse Robin problems, we establish a logarithmic convergence rate for Sobolev-regular Robin coefficients, whereas for analytic coefficients we can attain an algebraic rate. The use of rescaled analytic Gaussian priors in posterior consistency for nonlinear inverse problems is new and may be of separate interest in other inverse problems. Our numerical results illustrate the convergence property in two observation settings. 
\end{abstract}

\ifpdf
\hypersetup{pdftitle={robin},
  pdfauthor={A. K. Rasmussen, F. Seizilles, M. Girolami, I. Kazlauskaite}}
\fi

\maketitle

\section{Introduction}
The Bayesian approach has in recent years gained traction as a powerful and flexible framework for solving inverse problems by allowing the user to model prior knowledge, regularize reconstructions and quantify uncertainty, see \cite{stuart2010}. In this paper, we investigate the Bayesian approach to an emerging class of nonlinear inverse problems known as inverse Robin problems.\\

Inverse Robin problems appear in boundary value problems for partial differential equations (PDEs), where the boundary is partitioned into at least two parts: a hidden and an observable part. The hidden part carries information of a boundary effect modelled by a Robin boundary condition.
Then the Robin inverse problem is the inverse problem of recovering the Robin coefficient from Dirichlet and Neumann data on the observable part of the boundary. Our focus will be on the inverse Robin problem for a scalar Laplace equation and a Stokes' system of equations. The former is an inverse problem also known as corrosion detection and was considered in the early contribution \cite{inglese1997}. The latter appears when initialising ice sheet models for climate and sea-level predictions. This inverse problem asks for the unknown basal drag coefficient of the ice sediment from observations of ice velocity at the surface, see \cite{arthern2010}. Addressing this inverse problem in a statistical framework is a crucial step in improving the robustness and accuracy of ice sheet models for future sea-level projections.\\

Reconstruction for the inverse Robin problem for the Laplace equation has been studied using classical regularization methods based on penalized least squares, see \cite{chaabane1999,liu2017,jin2007} and the references therein. In \cite{inglese1997,fasino1999} accurate direct methods are provided given that the domain is sufficiently thin. The problem has been posed in a Bayesian framework in \cite{nicholson2022}, which determines the Robin coefficient and the hidden boundary simultaneously. The related inverse Robin problem for the Stokes PDE has been considered in the Bayesian framework in \cite{arthern2015,nicholson2018,babaniyi2021}, whereas in the two latter works the framework is similar to the general approach in \cite{stuart2010}.\\

Despite the success of the Bayesian approach to inverse problems, different paths of theoretical guarantees have been explored only recently. Statistical convergence analysis for the posterior distribution in nonlinear inverse problems has seen a recent interest with the framework devised in \cite{monard2021} based on the work in \cite{ghosal2000}, see also \cite{nickl2023}. In this approach, the main conditions of the forward map are that of forward regularity: the data should be uniformly bounded and depend continuously on the parameter given that it is sufficiently smooth, and conditional (inverse) stability: the inverse of the forward map is continuous when restricted to a sufficiently small subset of the range. For forward regularity, we require a certain smoothness of solutions of the governing equation near the observable part of the boundary. This can be achieved by classical techniques in PDEs. Inverse stability results, however, rarely come cheap and require in-depth knowledge of the inverse problem at hand. For the Stokes model we consider, some conditional stability results have been developed, see Theorem 1.5 and Remark 3.7 in \cite{boulakia2013b}, which quantifies the unique continuation result of \cite{fabre1996}. 
Common for the inverse Robin problems is the fact that the spatially varying Robin coefficient $\beta$ enters in a Robin condition of the form $\partial_\nu u + \beta u = 0$ at the hidden boundary, where $\nu$ is the outgoing unit normal and $u$ is the solution of the governing equation. 
So if $u$ is known and nonzero here, the reconstruction is a matter of algebra: $\beta = -u^{-1}\partial_\nu u$. However, determining conditions for which $u$ in a Stokes' model is nonzero on the hidden boundary remains a largely unsolved problem, see \cite{boulakia2013b}. For this reason we motivate our approach for general Robin-type inverse problems by the prototypical model for the scalar Laplace equation, see \cite{chaabane1999}. It is not uncommon that methods used in solving the Robin inverse problem for the Laplace equation have stimulated the development of approaches for solving the corresponding problem for the Stokes model, see \cite{arthern2010} which makes use of the Kohn-Vogelius functional \cite{kohn1987}.\\

Inverse Robin problems are related to the Cauchy problem of determining a solution to a Laplace equation in a domain from Cauchy data on parts or the whole of the boundary. This is because the `known' in our inverse Robin problems consists of Cauchy data on a part of the boundary. 
The global Cauchy problem of determining the solution in the entire domain is known to be severely ill-posed since \cite{hadamard1953}. Conditional stability estimates of logarithmic kind exist for this global problem \cite[Theorem 1.9]{alessandrini2009,chaabane2004}, while for stability in the interior, Hölder estimates can be obtained, see Theorem 1.7 and Remark 1.8 of \cite{alessandrini2009}. Combining the latter with analytic continuation for `uniformly' analytic (in the sense of $\mathcal{R}_2(M)$ defined below) solutions near the hidden boundary, one obtains conditional Hölder stability for the inverse Robin problem. 
This is essentially the content of \cite{hu2015}, which we modify to our setting below in Lemma \ref{lemma:stability}. We note that in \cite{sincich2007} a Hölder stability estimate is obtained for the scalar Laplace equation, when the Robin coefficient is piecewise constant on \textit{a priori} known sets. 
Using properties of the derivative of the forward map, a Lipschitz stability for the inverse problem in Stokes' model has been established in \cite{egloffe2013} given that the Robin coefficient belongs to compact and convex subset of a finite dimensional vector space. Unlike \cite{boulakia2013b} it provides an estimate independent of observations of the pressure.\\

For inverse problems with regular forward maps and conditional stability estimates, guaranteeing the statistical convergence of the posterior distribution reduce to the choice of prior \cite{monard2021}. In this paper, we consider two classes of numerically tractable and popular choices of Gaussian process priors that fit the stability regimes of the inverse problem. Our theoretical results are two-fold: For a Matérn type Gaussian prior we show that as the number of observations increases, the posterior mean converges in probability to the true Robin coefficient given this has some Sobolev smoothness. Not surprisingly the convergence rate is logarithmic. On the other hand, if the Robin coefficient is analytic, the squared exponential Gaussian prior provides an algebraic rate of convergence. This is the content of Theorem \ref{thm:main-theorem} below. Our approach for the analytic set of parameters follows the approach in \cite{monard2021,ghosal2000,nickl2023} using results in \cite{vaart2009}. Unlike \cite{vaart2009}, which considers analytic Gaussian processes with a change at time scale, we consider `rescaled' (in the sense of \eqref{eq:prior-def}) Gaussian processes. In Lemma \ref{lemma:analytic-consistency} we show that such priors satisfy the usual conditions for posterior consistency. This result is new and may be of independent interest in other inverse problems when modelling analytic functions. This seems to be a useful \textit{a  priori} class of functions to consider for some stability estimates, see for example \cite{labreuche1999}.\\

We perform numerical experiments on the Laplace and Stokes models using Markov chain Monte Carlo (MCMC) methods and finite element methods. The experiments support the theoretical findings on the improvements in the reconstruction of the Robin coefficients as the number of observations increases.\\

In Section \ref{sec:inverse-robin-problems} we give the setting of two inverse Robin problems in the context of a Stokes system of PDEs and a Laplace equation in two dimensions. We state results on the regularity properties of the forward maps as well as conditional stability estimates. When not relying on existing results, these are proved in Appendix \ref{sec:forward-reg} and \ref{sec:cond-stab-est}. In Section \ref{sec:bayesian-approach} we recap the Bayesian approach to inverse problems, describe the observation model, and present the Matérn-type and squared exponential Gaussian priors. Here we also give the main result, Theorem \ref{thm:main-theorem}, which is proved in Appendix \ref{sec:cons-analytic}. In Section \ref{sec:experiments}, we present the numerical approach and results for the scalar Laplace model and the Stokes model.\\

In the following, we let random variables be defined on a probability space $(\Omega,\mathcal{F},\mathrm{Pr})$. For a metric space $\mathcal{X}$ the Borel $\sigma$-algebra is denoted by $\mathcal{B}(\mathcal{X})$. Given a random element $X:\Omega\rightarrow \mathcal{X}$ that is $\mathcal{F}-\mathcal{B}(\mathcal{X})$ measurable, we denote its \textit{law} or \textit{distribution} by the probability measure $\mathcal{L}(X)$ defined by $\mathcal{L}(X)(B)=\mathrm{Pr}(X^{-1}(B))$ for all $B\in \mathcal{B}(\mathcal{X})$.
\\
\section{Inverse Robin problems}\label{sec:inverse-robin-problems}
\subsection{Stokes' model}
We consider the constant viscosity Stokes ice sheet model for a velocity field $u: \mathcal{O}\rightarrow \mathbb{R}^d$ and pressure $p:\mathcal{O} \rightarrow \mathbb{R}$ in a bounded and smooth domain $\mathcal{O} \subset \mathbb{R}^d$, $d=2,3$,
\begin{equation}\label{eq:stokes}
    \begin{aligned}
    -\Delta u + \nabla p &= \rho g \qquad &&\text{ in $\mathcal{O}$},\\
    \nabla \cdot u &= 0 &&\text{ in $\mathcal{O}$},\\
    \partial_{\nu} u - p\nu &= h &&\text{ on $\Gamma_s$},\\
    \partial_{\nu} u - p\nu + \beta u &= 0 &&\text{ on $\Gamma_\beta$},
\end{aligned}
\end{equation}
where $\nu$ is the outward unit normal, $\rho$ is a constant density of the ice, $g$ is the gravitational field and $h$ is the prescribed boundary stress. Here $\Gamma_s$ and $\Gamma_\beta$ are disjoint and connected open subsets of the boundary such that $\partial \mathcal{O} = \overline{\Gamma}_s \cup \overline{\Gamma}_\beta$. We denote by $\Gamma$ an open subset of $\Gamma_s$, where we make our measurements. The Robin inverse problem is then to recover $\beta$ given $u|_\Gamma$, that is, to invert the nonlinear forward map
\[G: \beta\mapsto u|_\Gamma.\]
For physical accuracy, we assume $\beta$ is a positive function, $\beta\geq m_\beta>0$. We reparametrize the forward map to
\[\mathcal{G}(\theta):=G(m_\beta+e^\theta)=u|_\Gamma\]
defined on our parameter space 
\[\Theta:=H^1(\Gamma_\beta).\]
The choice of the parameter space $\Theta$ makes $\theta \mapsto \mathcal{G}(\theta)$ continuous into $(C(\overline{\Gamma}))^2$, which, as we shall see in Section \ref{sec:bayesian-approach}, leaves us with a well-defined posterior distribution. It follows from Lax-Milgram theory in the Hilbert space of divergence-free $(H^1(\mathcal{O}))^d$-functions that there exists a unique solution $u\in (H^1(\mathcal{O}))^d$ to \eqref{eq:stokes} for any positive and bounded $\beta$, hence the forward map is well-defined. Further, when $\beta$ is continuous, unique continuation results \cite[Corollary 1.2]{boulakia2013b} imply injectivity of $G$, see for example \cite[Proposition 3.3]{boulakia2013a}.  These facts are proven in the following lemma for the case $d=2$. For $d=3$ this follows in the same way, but for example for the choice $\Theta=H^2(\Gamma_\beta)$.

\begin{lemma}\label{lemma:well-posedness-stokes}
    Let $h\in (L^2(\Gamma_s))^2$, $\rho g\in (L^2(\mathcal{O}))^2$ and $\theta,\theta_1,\theta_2\in L^\infty(\Gamma_\beta)$. We have the following:
    \begin{enumerate}[label=(\roman*)]
    \item Set $\beta=\beta(\theta):=m_\beta+e^\theta$. Then there is a unique solution $u\in H^1(\mathcal{O}))^2$ to \eqref{eq:stokes}.
    \item If $\|\theta_1\|_{H^1(\Gamma_\beta)},\|\theta_2\|_{H^1(\Gamma_\beta)}\leq M$ and $\Gamma\subset \subset \Gamma_s$, then there exists $\alpha>0$ such that
    $$\|\mathcal{G}(\theta_1)-\mathcal{G}(\theta_2)\|_{(C(\overline{\Gamma}))^2}\leq C(\mathcal{O},m_\beta,h,\rho,g) \|\theta_1-\theta_2\|_{H^{1}(\Gamma_\beta)}^\alpha.$$
    \item $\Theta\ni \theta \mapsto \mathcal{G}(\theta) \in (C(\overline{\Gamma}))^2$ is injective.
\end{enumerate}
\end{lemma}

\subsection{Scalar Laplace equation}
Consider the following Laplace equation for a potential $u:\mathcal{O}\rightarrow \mathbb{R}$, surface normal current $h\in H^{-1/2}(\Gamma)$ and Robin coefficient $\beta\in L^\infty(\Gamma_\beta)$,
\begin{equation}\label{eq:laplace}
    \begin{aligned}
    \Delta u &= 0 \qquad &&\text{ in $\mathcal{O}$},\\
    \partial_{\nu} u &= h &&\text{ on $\Gamma$},\\
    u &= 0 && \text{ on $\Gamma_0$},\\
    \partial_{\nu} u + \beta u &= 0 &&\text{ on $\Gamma_\beta$},
\end{aligned}
\end{equation}
where $\partial \mathcal{O} = \overline{\Gamma} \cup \overline{\Gamma}_0 \cup \overline{\Gamma}_\beta$. Here a homogeneous Dirichlet condition is introduced for the stability estimate in \cite{alessandrini2003}, which we use in Lemma \ref{lemma:stability} below. As before our goal is to invert the reparametrized forward map
\begin{equation}\label{eq:forward-map}
   \mathcal{G}(\theta):=G(m_\beta+e^{\theta})=u|_\Gamma,
\end{equation}
where we with a slight misuse of notation keep the notation $G$ and $\mathcal{G}$ for this model.\\

\begin{figure}
\centering
\tikzset{every picture/.style={line width=0.75pt}} 

\scalebox{0.75}{
\begin{tikzpicture}[x=0.75pt,y=0.75pt,yscale=-1,xscale=1]

\draw  [fill={rgb, 255:red, 74; green, 144; blue, 226 }  ,fill opacity=0.34 ] (107,106) .. controls (107,92.75) and (117.75,82) .. (131,82) -- (313,82) .. controls (326.25,82) and (337,92.75) .. (337,106) -- (337,178) .. controls (337,191.25) and (326.25,202) .. (313,202) -- (131,202) .. controls (117.75,202) and (107,191.25) .. (107,178) -- cycle ;
\draw [color={rgb, 255:red, 155; green, 155; blue, 155 }  ,draw opacity=1 ][line width=2.25]    (131,202) .. controls (103,201) and (107,178) .. (107,144) .. controls (107,110) and (102,82) .. (131,82) ;
\draw [color={rgb, 255:red, 155; green, 155; blue, 155 }  ,draw opacity=1 ][line width=2.25]    (313,202) .. controls (342,201) and (337,179) .. (337,145) .. controls (337,111) and (343,82) .. (313,82) ;
\draw [color={rgb, 255:red, 0; green, 58; blue, 255 }  ,draw opacity=1 ][line width=2.25]    (131,82) -- (313,82) ;
\draw [color={rgb, 255:red, 0; green, 0; blue, 0 }  ,draw opacity=1 ][line width=2.25]    (131,202) -- (313,202) ;
\draw [color={rgb, 255:red, 128; green, 128; blue, 128 }  ,draw opacity=1 ]   (390,201) -- (418,201) ;
\draw [shift={(420,201)}, rotate = 180] [color={rgb, 255:red, 128; green, 128; blue, 128 }  ,draw opacity=1 ][line width=0.75]    (4.37,-1.96) .. controls (2.78,-0.92) and (1.32,-0.27) .. (0,0) .. controls (1.32,0.27) and (2.78,0.92) .. (4.37,1.96)   ;
\draw [color={rgb, 255:red, 128; green, 128; blue, 128 }  ,draw opacity=1 ]   (390,201) -- (390,172) ;
\draw [shift={(390,170)}, rotate = 90] [color={rgb, 255:red, 128; green, 128; blue, 128 }  ,draw opacity=1 ][line width=0.75]    (4.37,-1.96) .. controls (2.78,-0.92) and (1.32,-0.27) .. (0,0) .. controls (1.32,0.27) and (2.78,0.92) .. (4.37,1.96)   ;

\draw (211,132) node [anchor=north west][inner sep=0.75pt]   [align=left] {$\displaystyle \mathcal{O}$};
\draw (213,55.4) node [anchor=north west][inner sep=0.75pt]    {$\textcolor[rgb]{0,0.23,1}{\Gamma }\textcolor[rgb]{0,0.23,1}{_{s}}$};
\draw (211,209.4) node [anchor=north west][inner sep=0.75pt]    {$\Gamma _{\beta }$};
\draw (400,202.4) node [anchor=north west][inner sep=0.75pt]  [font=\footnotesize]  {$\textcolor[rgb]{0.5,0.5,0.5}{x}$};
\draw (375,182.4) node [anchor=north west][inner sep=0.75pt]  [font=\footnotesize]  {$\textcolor[rgb]{0.5,0.5,0.5}{y}$};
\draw (77,131.4) node [anchor=north west][inner sep=0.75pt]  [color={rgb, 255:red, 155; green, 155; blue, 155 }  ,opacity=1 ]  {$\Gamma _{0}$};
\draw (355,132.4) node [anchor=north west][inner sep=0.75pt]  [color={rgb, 255:red, 155; green, 155; blue, 155 }  ,opacity=1 ]  {$\Gamma _{0}$};

\end{tikzpicture}
}
\caption{Diagram of the domain}
\label{fig:domain}
\end{figure}
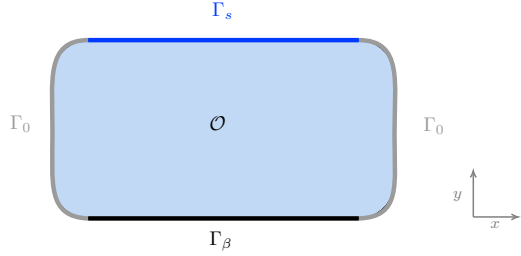

\begin{assumption}[Domain]\label{assump:domain}
    We assume $\Gamma_\beta = (0,1) \times \{0\}$ and define $\Gamma_{\beta,\epsilon}:=(\epsilon,1-\epsilon)\times \{0\}$ for some $0<\epsilon<1$. Furthermore, we assume
    $\partial \mathcal{O}$ is a simple closed curve decomposed into four subarcs oriented as $\Gamma_\beta$, $\Gamma_0^1$, $\Gamma$, $\Gamma_0^2$, and where $\Gamma_0 = \Gamma_0^1\cup \Gamma_0^2$.
\end{assumption}
We assume $\Gamma_\beta=(0,1) \times \{0\}$ since we want to avoid defining Gaussian processes on manifolds. Occasionally, we identify $\Gamma_\beta$ with $(0,1)\subset \mathbb{R}$. For the two stability estimates in Lemma \ref{lemma:stability} we could generalize to a $C^2$ or analytic $\Gamma_\beta$, respectively. To analyze the stability of the forward map we find it useful to restrict it to two well-chosen bounded and closed subsets of $\Theta$.
\begin{assumption}\label{assump:beta}
Assume first $\beta = \beta(\theta) := m_\beta + e^{\theta}$ for $\theta \in \Theta$ with
$$\Theta := H^1(\Gamma_\beta).$$
Depending on the setting we accept either of the two following assumptions for some $M>0$:
\begin{enumerate}[label=(\roman*)]
    \item Assume $\theta \in \mathcal{R}_1(M)$ with
    $$\mathcal{R}_1(M):=\{\theta \in H^1(\Gamma_\beta): \|\theta\|_{H^1(\Gamma_\beta)}\leq M\}$$
    \item Assume $\theta \in \mathcal{R}_2(M)$ with
    $$ \mathcal{R}_2(M):=\{\theta \in C^\infty(\overline{\Gamma}_\beta): \|\theta\|_{L^\infty(\Gamma_\beta)}\leq M, \sup_{x\in \overline{\Gamma}_\beta}|(\partial^k\beta)(x)|\leq M (k!)M^k\}$$
\end{enumerate}
\end{assumption}
It is well-known that $\beta$ is analytic on $\overline{\Gamma}_\beta$ if and only if $\beta \in C^\infty(\overline{\Gamma}_\beta)$ with 
$$\sup_{x \in \overline{\Gamma}_\beta} |(\partial^k \beta)(x)|\leq M_\beta(k!)M_\beta^k$$ 
for some $M_\beta>0$, see \cite[chapter 1]{krantz2002}. We can think of $\mathcal{R}_2(M)$ as functions that are `uniformly' analytic with the added condition $\|\theta\|_{L^\infty(\Gamma_\beta)}\leq M$ to ensure $\theta\mapsto \beta$ is continuous in both directions. The sets $\mathcal{R}_1(M)$ and $\mathcal{R}_2(M)$ are exactly the `regularization sets' for which stability results for the inverse problem are available, see Lemma \ref{lemma:stability}. To this end, we make the following assumption.

\begin{assumption}\label{assump:h}
We assume that $h$ is not identical to a constant and that
$h \in H_1:=\{h\in H^{1/2}(\Gamma): h\geq 0, \|h\|_{H^{1/2}}\leq M_h\}$ for some $M_h>0$.
\end{assumption}
The positivity assumption is only needed for the stability estimate stated in Lemma \ref{lemma:stability} (ii), and it might be avoided as in \cite{alessandrini2003}. In the following we prove a number of auxiliary results, where we specify sufficient conditions on $\theta$ and $\beta$. First, we note that the forward map is well-defined.
\begin{lemma}\label{lemma:classicH1}
For $\beta=\beta(\theta)$ with $\theta\in L^\infty(\Gamma_\beta)$ and $h\in H^{-1/2}(\Gamma)$, there exists a unique solution $u \in H^1(\mathcal{O})$ of \eqref{eq:laplace} with
\begin{equation}\label{eq:u-estH1}
    \|u\|_{H^1(\mathcal{O})} \leq C\|h\|_{H^{-1/2}(\partial \mathcal{O})},
\end{equation}
for some constant $C=C(\mathcal{O},m_\beta)>0$. 
\end{lemma}
Secondly, the forward map is Lipschitz continuous on certain bounded sets of $\Theta$ and the observations are uniformly bounded.
\begin{lemma}\label{lemma:forward-reg}
    Let $h\in H_1$ and $\theta_1,\theta_2 \in L^\infty(\Gamma_\beta)$. Then,
    \begin{enumerate}[label=(\roman*)]
    \item if $\beta=\beta(\theta_1)$ and $\|\beta(\theta_1)\|_{H^1(\Gamma_\beta)}\leq M$ we have
    $$\|\mathcal{G}(\theta_1)\|_{C(\overline{\Gamma})}\leq U(\mathcal{O},m_\beta,M,M_h),$$
    \item if $\beta_i=\beta(\theta_i)$ for $i=1,2$ with $\|\theta_1\|_{L^\infty(\Gamma_\beta)},\|\theta_2\|_{L^\infty(\Gamma_\beta)}\leq M$ then 
    $$\|\mathcal{G}(\theta_1)-\mathcal{G}(\theta_2)\|_{L^2(\Gamma)}\leq K(\mathcal{O},m_\beta,M)\|\theta_1-\theta_2\|_{L^\infty(\Gamma_\beta)}.$$
    \item if $\theta_1,\theta_2\in \mathcal{R}_1(M)$ then there exists an $\alpha>0$ such that
    $$\|\mathcal{G}(\theta_1)-\mathcal{G}(\theta_2)\|_{C(\overline{\Gamma})}\leq C(\mathcal{O},m_\beta,M)\|\theta-\theta'\|_{L^\infty(\Gamma_\beta)}^\alpha.$$
\end{enumerate}
\end{lemma}
The following result is that of conditional (inverse) stability, where the condition is either $\theta\in \mathcal{R}_1(M)$ or $\theta\in \mathcal{R}_2(M)$.
\begin{lemma}[Conditional stability]\label{lemma:stability}
Let $\mathcal{O}$ satisfy Assumption \ref{assump:domain} and $h$ satisfy Assumption \ref{assump:h}. 
 \begin{enumerate}[label=(\roman*)]
    \item If $\theta_i\in \mathcal{R}_1(M)$, $i=1,2,$ then there exists constants $K_1>0$ and $0<\sigma_1<1$ such that
    $$\|\theta_1-\theta_2\|_{L^\infty(\Gamma_{\beta,\epsilon})}\leq K_1 | \log (\|\mathcal{G}(\theta_1)-\mathcal{G}(\theta_2)\|_{L^2(\Gamma)}) |^{-\sigma_1},$$
    where $K_1$ and $\sigma_1$ depend only on $\mathcal{O}$, $h$, $m_\beta$, $M$ and $\epsilon$.
    \item If $\theta_i\in \mathcal{R}_2(M)$, $i=1,2$, then there exists constants $K_2>0$ and $0<\sigma_2<1$ such that
    $$\|\theta_1-\theta_2\|_{L^2(\Gamma_{\beta,\epsilon})}\leq K_2\|\mathcal{G}(\theta_1)-\mathcal{G}(\theta_2)\|_{L^2(\Gamma)}^{\sigma_2},$$
    where $K_2$ and $\sigma_2$ depend  $\mathcal{O}$, $M_h$, $M$ and $\epsilon$.
 \end{enumerate}
\end{lemma}
The stability result $(ii)$ generalizes to three dimensions. In this case, one technical obstacle is to analyze the smoothness of the solutions near the corner singularities. 
\section{The Bayesian approach}\label{sec:bayesian-approach}
Central in the Bayesian framework is the posterior distribution, which is the normalized product of the prior distribution and the likelihood-function modelling the measurement process. In this paper we take the natural viewpoint of \cite{nickl2023} that the measurements are discrete, taken at uniformly random locations on the observable part of the boundary, and are contaminated with Gaussian noise. In the context of Stokes' model, we let $V=\mathbb{R}^2$ and $d=2$, whereas for the Laplace equation we set $V=\mathbb{R}$ and $d=1$. In both cases we let $|\cdot|_V$ denote the Euclidean norm. Our observations arise as the sequence of random vectors $D_N := (Y_i,X_i)_{i=1}^N$ in $(V\times \Gamma)^N$ of the form
\begin{equation}\label{eq:observations}
    Y_i=\mathcal{G}(\theta)(X_i)+\varepsilon_i, \quad \varepsilon_i \overset{\mathrm{iid}}{\sim} N(0,1), \quad i=1,\hdots,N,
\end{equation}
where $X_i\overset{\mathrm{iid}}{\sim} \lambda$, the uniform distribution on $\Gamma$ independent of the noise $\varepsilon_i$. More precisely, we endow $\Gamma$ with a Borel $\sigma$-algebra $\mathcal{B}(\Gamma)$ generated by the open sets in $\Gamma$ with respect to arc length metric. 
We have $\mu(B) = |\Gamma|^{-1} \int_B \, dS$, where $dS$ is the usual length measure and $|\Gamma|=\int_\Gamma \, dS$. 

The random vectors $(Y_i,X_i)$ are i.i.d, and we denote their law $P_{\theta}$ with corresponding probability density (Radon-Nikodym derivative) 
$$p_\theta(y,x)\equiv\frac{dP_\theta}{d\mu}(y,x)=\frac{1}{(2\pi)^{d}}\exp\left (-\frac{1}{2}|y-\mathcal{G}(\theta)(x)|_V^2 \right), \quad y\in V, x\in \Gamma,$$
with respect to $d\mu = dy \times d\lambda$, where $dy$ is the Lebesgue measure on $V$. We call $\theta \mapsto p_\theta(y,x)$ the likelihood function, and denote by $P_{\theta}^N$ the joint law of the random variables $(Y_i,X_i)_{i=1}^N$. The likelihood function is suitable to enter in the Bayesian approach: Lemma \ref{lemma:well-posedness-stokes} and \ref{lemma:forward-reg} implies that $x\mapsto\mathcal{G}(\theta)(x)$ is continuous and that $\Theta \ni \theta \mapsto \mathcal{G}(\theta) \in C(\overline{\Gamma})^d$, where $d=2$ for Stokes' model and $d=1$ for the Laplace equation. This implies $(\theta,x)\mapsto \mathcal{G}(\theta)(x)$ is jointly $\mathcal{B}(\Theta)\otimes \mathcal{B}(\Gamma) - \mathcal{B}(V)$ measurable by Lemma 4.5.1 in \cite{aliprantis2006}, 
which is enough for a well-defined posterior distribution, see \cite{nickl2023}.

Given a prior distribution $\Pi$ supported in $\Theta$, Bayes' formula, see \cite[p. 7]{ghosal2017} or \cite{nickl2023}, updates $\Pi$ by the likelihood function to obtain the posterior distribution $\Pi(\cdot|D_N)$ of $\theta$ given $D_N$,
\begin{equation}\label{eq:posterior}
    \Pi(B|D_N) = \frac{\int_B e^{\ell_N(\theta)} \, \Pi(d\theta)}{\int_\Theta e^{\ell_N(\theta)} \, \Pi(d\theta)}, \quad B\in \mathcal{B}(\Theta),
\end{equation}
where
$$\ell_N(\theta):=-\frac{1}{2}\sum_{i=1}^N |Y_i-\mathcal{G}(\theta)(X_i)|_V^2.$$
Note that $0\leq |y-\mathcal{G}(\theta)(x)|^2<\infty$ for all $(y,x)\in V\times \Gamma$ and $\theta \in \Theta$, and hence the normalization constant satisfies
$$0<\int_\Theta e^{-\frac{1}{2}\sum_{i=1}^N|y_i-\mathcal{G}(\theta)(x_i)|_V^2} \, \Pi(d\theta)\leq 1$$
for all $(y_i,x_i)_{i=1}^N \in (V\times \Gamma)^N$. It follows that $B\mapsto \Pi(B|D_N)$ is a measure for each $D_N\in (V\times \Gamma)^N$ and that $\omega\mapsto \Pi(B|D_N(\omega))$ is measurable for every $B \in \mathcal{B}(\Theta)$. In particular, $\omega \mapsto \Pi(B|Y(\omega))$ is a $[0,1]$-valued random variable.
Before we state our main theorem on the convergence features of the posterior distribution, we specify our choice of prior distributions.
\subsection{Choice of prior}
In this section we recall well-known prior distributions that are supported in $\mathcal{R}_j(M)$, $j=1,2$, allowing us to make use of the stability estimates in Lemma \ref{lemma:stability}. Our focus will be on the Matérn-type and squared exponential Gaussian priors. For simplicity we define the Gaussian priors on the $[-\pi,\pi)$-torus $\mathbb{T}$ and restrict to $\Gamma_\beta$ when necessary. Note any torus in which $\Gamma_\beta$ is embedded is relevant and can be used. In the case of the Matérn priors, as we shall see, this allows us to recover any sufficiently regular Sobolev function defined on $\Gamma_\beta$. On the other hand, the squared exponential Gaussian processes allows us to recover analytic functions defined on $\Gamma_\beta$ whose extension is $2\pi$-periodic. This setting benefits from the fact that properties of Sobolev regularity and analyticity of periodic functions are straightforwardly characterized by a decay of the Fourier coefficients. We can think of this setting as an implicit choice of approximation of the ground truth by the periodic trigonometric functions.
One could instead define a prior distribution on $\mathbb{R}$ with exponentially decaying spectral measure, and show that it is supported in $\mathcal{R}_2(M)$, see \cite{vaart2009}. This can be more technical due to the non-compactness of $\mathbb{R}$ and is unnecessary for our case.

Consider the usual $L^2(\mathbb{T})$ real orthonormal basis of trigonometric functions $\{\phi_k\}_{k\in \mathbb{Z}}$ and for $j=1,2$ the random series
\begin{equation}\label{eq:random-series}
    \tilde{\theta}_j = \sum_{k \in \mathbb{Z}} g_k w_{k,j}\phi_k, \qquad g_k \stackrel{i.i.d}{\sim} N(0,1)
\end{equation}
with
\begin{align}
    w_{k,1} &= (1+k^2)^{-\alpha/2}, \quad \alpha > 1/2, \label{eq:matern}\\
    w_{k,2} &= e^{-\frac{r}{2} k^2}, \quad r >0, \label{eq:squared_exp}
\end{align}
where $\alpha>0$ and $r>0$ are parameters to be chosen. We consider for example $\phi_k(x)=1/\sqrt{\pi}\cos(kx)$ for $k>0$, $\phi_k(x)=1/\sqrt{\pi}\sin(kx)$ for $k<0$ and $\phi_0=1/{\sqrt{2\pi}}$.
Since $w_{k,j}\in \ell^2(\mathbb{Z})$ for $j=1,2$, the series \eqref{eq:random-series} converges for each $x\in \mathbb{T}$ in the mean-square sense. In fact it is a Gaussian random variable, see \cite[p. 13]{ghosal2017}, and the limit $\tilde{\theta}_j(x)$ exists almost surely. The choice of $w_{k,j}$ is here motivated by the span of $\{w_{k,j}\phi_k\}_{k\in \mathbb{Z}}$. Indeed, $\{w_{k,1}\phi_k\}_{k\in \mathbb{Z}}$ is an orthonormal basis of 
\begin{equation}
    H^{\alpha}(\mathbb{T}):=\{f\in L^2(\mathbb{T}): \|f\|_{H^\alpha,\mathbb{T}}^2:= \sum_{k\in \mathbb{Z}} |f_k|^2 (1+k^2)^{\alpha} <\infty\}.
\end{equation}
 Here $f_k:=\langle f, \phi_k\rangle_{L^2(\mathbb{T})}$ denotes the coefficients in the orthonormal basis. Note we can write $f=\sum f_k \phi_k$ in a standard complex Fourier expansion $\sum \hat{f}_k e^{ikx}$ with the usual Fourier coefficients $\hat{f}_k$ expressed in terms of $f_k$. Conversely, any real function in the standard complex Fourier expansion can be written as $\sum f_k \phi_k$. Then $H^\alpha(\mathbb{T})$ is the usual periodic Sobolev space of regularity $\alpha$, see \cite{taylor2011}.
 Similarly, $\{w_{k,2}\phi_k\}_{k\in \mathbb{Z}}$ is an orthonormal basis of 
\begin{equation}
    \mathcal{A}_r(\mathbb{T}):=\{f\in L^2(\mathbb{T}): \|f\|_{r,\mathbb{T}}^2:= \sum_{k\in \mathbb{Z}} |f_k|^2 e^{rk^2} <\infty\}.
\end{equation}
A closed ball in any space of functions with exponentially decaying Fourier coefficients is in $\mathcal{R}_2(M)$ for some $M>0$, see Lemma \ref{lemma:fourier-analytic}, and so the choice of the `square' here is only in honor of the squared exponential prior. Note both spaces are Hilbert spaces as closed subspaces of $L^2(\mathbb{T})$ with their respective obvious inner products. Note also that $\mathcal{A}_r(\mathbb{T})$ embeds continuously into $H^\alpha(\mathbb{T})$ for any $r,\alpha > 0$, which in return embeds continuously into $C(\mathbb{T})$ for $\alpha > 1/2$ by a Sobolev embedding, see \cite{taylor2011}. 

\subsubsection{RKHS and support}\label{sec:RKHS}
The random series \eqref{eq:random-series} converges almost surely in $H^\beta(\mathbb{T})$ with $\beta<\alpha-\frac{1}{2}$ and $\mathcal{A}_q$ with $q<r$  for $j=1$ and $j=2$, respectively. Indeed, by Fubini's theorem
$$\mathbb{E}[\|\tilde{\theta}_1\|_{H^\beta,\mathbb{T}}^2]=\mathbb{E}[\sum_{k\in\mathbb{Z}} g_k^2 w_{k,1}^2(1+k^2)^{\beta}]=\sum_{k\in\mathbb{Z}}(1+k^2)^{\beta-\alpha}<\infty,$$
and similarly for $\tilde{\theta}_2$. Then also $\tilde{\theta}_2\in H^\beta(\Gamma_\beta)$ almost surely. Likewise we define 
$$\mathcal{A}_r(\Gamma_\beta):=\{f=g|_{\Gamma_\beta}: g\in \mathcal{A}_r(\mathbb{T})\},$$
endowed with the quotient norm
\begin{equation}\label{eq:quotient-norm}
    \|f\|_r=\inf_{g\in \mathcal{A}_r(\mathbb{T}): g=f \text{ in } \Gamma_\beta} \|g\|_{r,\mathbb{T}}=\|f\|_{r,\mathbb{T}},
\end{equation}
where the last equality holds because $f$ has a unique analytic continuation to $\mathbb{T}$. Then $\tilde{\theta}_2\in\mathcal{A}_q(\Gamma_\beta)$, $q<r$ almost surely.\\

The series \eqref{eq:random-series} is the Karhunen-Loeve expansion of a Gaussian random element of $H^\beta(\mathbb{T})$ and $\mathcal{A}_q(\mathbb{T})$ for $j=1$ and $j=2$, respectively, see \cite{dashti2017}. We set $\alpha>3/2$ and $r>0$ such that the laws of $\tilde{\theta}_1$ and $\tilde{\theta}_2$ define Gaussian probability measures in $\Theta$. By a Sobolev embedding $\tilde{\theta}_1$ and $\tilde{\theta}_2$ are almost surely in $C(\overline{\Gamma}_\beta)$, the separable Banach space
of continuous functions on $\overline{\Gamma}_\beta$ endowed with the usual supremum norm, which we denote by $\|\cdot\|_\infty$. Then the laws of $\tilde{\theta}_j$, $j=1,2$, define Gaussian probability measures on $C(\overline{\Gamma}_\beta)$, see \cite[Exercise 3.39]{hairer2009}. We denote
$$\tilde{\Pi}_j := \mathcal{L}(\tilde{\theta}_j), \quad j=1,2.$$
The reproducing kernel Hilbert space (RKHS) of the Gaussian random element $\tilde{\theta}_j$ is $H^\alpha(\mathbb{T})$ for $j=1$ and $\mathcal{A}_r(\mathbb{T})$ for $j=2$, see Theorem I.23 \cite{ghosal2017}. Since the restriction $H^\alpha(\mathbb{T})\rightarrow H^\alpha(\Gamma_\beta)$ is onto, see \cite[Section 4.4]{taylor2011}, the RKHS of the restricted Gaussian random element is $\mathcal{H}_1:=H^\alpha(\Gamma_\beta)$ in the case $j=1$ and $\mathcal{H}_2:=\mathcal{A}_r(\Gamma_\beta)$, see \cite[Exercise 2.6.5]{gine2016}.

\subsubsection{Covariance function} 
Since $\tilde{\theta}_j(x)$ is a Gaussian random variable for each $x\in \mathbb{T}$ and $j=1,2$, it is in fact a Gaussian process. The covariance function $K_j:\mathbb{T}\times \mathbb{T}\rightarrow \mathbb{R}$ of the process takes the form, for $j=1,2$,
\begin{equation}
    K_j(x,x') = \mathbb{E}[\tilde{\theta}_j(x)\tilde{\theta}_j(x')] = \sum_{k\in\mathbb{Z}} w_{k,j}^2 \phi_k(x)\phi_k(x'),
\end{equation}
see for example \cite[p. 586]{ghosal2017}. Choosing for example $\phi_k(x)=1/\sqrt{\pi}\cos(kx)$ for $k>0$, $\phi_k(x)=1/\sqrt{\pi}\sin(kx)$ for $k<0$ and $\phi_0=1/{\sqrt{2\pi}}$, and using the identity $\cos(a)\cos(b)+\sin(a)\sin(b)=\cos(a-b)$ we find
\begin{align}
    K_j(x,x') &= \frac{w_{0,j}^2}{2\pi} + \frac{1}{\sqrt{\pi}}\sum_{k=1}^\infty w_{k,j}^2  \phi_k(x-x'),\\
    &= \frac{1}{2\pi}\sum_{k\in\mathbb{Z}} w_{k,j}^2e^{ik(x-x')},\\
    &=\sum_{k\in\mathbb{Z}} m_{j}(x-x'+2\pi k),
\end{align}
using the Poisson summation formula with 
\begin{align}
    m_1(s)&= \mathcal{F}^{-1}[(1+4\pi^2 \xi^2)^{-\alpha}](s)=C s^{\alpha-1/2}\mathcal{K}_{\alpha-\frac{1}{2}}(s),\\
    m_2(s)&= \mathcal{F}^{-1}[e^{-4\pi^2 r\xi^2}](s) = C e^{-\frac{s^2}{4r^2}},
\end{align}
where $\mathcal{K}_\nu$, $\nu>0$ is a modified Bessel function, see \cite[Section 4.2.1]{rasmussen2006}. Thus $K_j$ is the $2\pi$-periodization of the usual Matérn covariance function on $\mathbb{R}$ when $j=1$ and the squared exponential covariance functions on $\mathbb{R}$ when $j=2$, which justifies our naming convention.
\subsubsection{Rescaling}
Take $\alpha>1$ and $r>0$ such that $\tilde{\Pi}_j(\Theta)=1$ for $j=1,2$. We then let $\Pi_j$ be the `rescaled' Gaussian distribution for $j=1,2$, 
\begin{equation}\label{eq:prior-def}
    \Pi_j := \mathcal{L}\left(\kappa_{N,j} \tilde{\theta}_j\right), \qquad \tilde{\theta}_j\sim \tilde{\Pi}_j,
\end{equation}
for some decreasing sequence in $N$, $\kappa_{N,j}$ defined as
\begin{align}
    \kappa_{N,1} &:= N^{-1/(4\alpha+2)},\\
    \kappa_{N,2} &:= \log(N)^{-1}.
\end{align}
Letting the covariance of the prior depend on the observation regime is natural: it updates the weight of the prior term in the posterior \eqref{eq:posterior} formally as
\[d\Pi_1(\theta)\propto \exp\left(-\frac{N^{1/(2\alpha+1)}}{2}\|\theta\|_{\mathcal{H}_1}^2 \right)\]
in the case of $j=1$. In this way we penalize large values of $\|\theta\|_{\mathcal{H}_1}$ more. This is common in the consistency literature, see \cite{monard2021}, and in fact sufficient for convergence of regularized least-square procedures, see \cite[Section 5]{engl1996}. In our setting this rescaling is needed so that the prior distributions concentrate sufficiently on the totally bounded regularization sets $\mathcal{R}_1(M)$ and $\mathcal{R}_2(M)$.

\subsection{Convergence of the posterior mean}
Before we state the main result, the convergence of the posterior mean to the ground truth as $N\rightarrow \infty$, we recall some preparatory definitions. In the following we let $\Pi_j(\cdot|D_N)$ denote the posterior distribution \eqref{eq:posterior} in $\Theta$ arising from the prior distribution $\Pi_j$ defined in \eqref{eq:prior-def} for $j=1,2$. The posterior mean $\mathbb{E}_j[\theta|D_N]$ is defined in the sense of a Bochner integral, see for example \cite[p. 44]{diestel1977}. Indeed, for all $D_N\in (V\times \Gamma)^N$ 
\begin{equation}\label{eq:bocnermean}
    \int_\Theta \|\theta\|_{\Theta}\, d\Pi_j(\theta|D_N) \propto \int_\Theta \|\theta\|_{\Theta} e^{\ell_N(\theta)} \, d\Pi_j(\theta) \leq \int_\Theta \|\theta\|_{\Theta} \, d\Pi_j(\theta)  < \infty,
\end{equation}
by Fernique's theorem \cite[Theorem 3.11]{hairer2009}, since $\Pi_j$ is supported in $\Theta$ for $j=1,2$. Then $D_N\mapsto \mathbb{E}_j[\theta | D_N]$ is a $\Theta$-valued random element by the definition of the Bochner integral and since the pointwise limit of a sequence of measurable functions is measurable, see \cite[Theorem 4.2.2]{dudley1989}. Let $\epsilon_N>0$ be some decreasing sequence in $N$ converging to zero. We say that a sequence of real-valued random variables $\{f_N(D_N)\}_{N=1}^\infty$ converges to zero in $P_{\theta_0}^N$-probability with rate $\epsilon_N$ as $N\rightarrow \infty$ if there exists a constant $C>0$ such that
\begin{equation}
   \lim_{N\rightarrow \infty} P_{\theta_0}^N(D_N: |f_N(D_N)| > C\epsilon_N ) = 0
\end{equation}
Then we have the following convergence results for the reconstruction error of the posterior mean, where we take $f_N(D_N) = \|\mathbb{E}_j[\theta|D_N]-\theta_0\|$ for $j=1,2$, and a suitable norm $\|\cdot\|$.

\begin{theorem}\label{thm:main-theorem}
Consider the posterior distribution $\Pi_j(\,\cdot\,|D_N)$ arising from observations \eqref{eq:observations} in the model \eqref{eq:forward-map} and prior distributions $\Pi_{j}$, $j=1,2$. 
\begin{enumerate}[label=(\roman*)]
    \item If $\theta_0 \in H^\alpha(\Gamma_\beta)$, $\alpha>3/2$, then
    $$\|\mathbb{E}_1[\theta | D_N] - \theta_0\|_{L^\infty(\Gamma_{\beta,\epsilon})} \rightarrow 0 \qquad \text{in $P_{\theta_0}^N$-probability}$$
    with rate $|\log(C\delta_N)|^{-\sigma}$ as $N\rightarrow \infty$ for some $0<\sigma<1$ and constant $C>0$ and where $\delta_N=N^{-\alpha/(2\alpha+1)}$.
    \item If $\theta_0 \in \mathcal{A}_r(\Gamma_\beta)$, $r>0$,
    then
    $$\|\mathbb{E}_2[\theta | D_N] - \theta_0\|_{L^2(\Gamma_{\beta,\epsilon})} \rightarrow 0 \qquad \text{in $P_{\theta_0}^N$-probability}$$
    with rate $\delta_N^\sigma$ for some $0<\sigma<1$  as $N\rightarrow \infty$, and where $\delta_N=N^{-1/2}\log(N)$.
\end{enumerate}
\end{theorem}
\begin{proof}
    $(i)$ This is the result of Theorem 2.3.2 \cite{nickl2023} and \cite[Exercise 2.4.4]{nickl2023} whose conditions are satisfied by Lemma \ref{lemma:forward-reg}, \ref{lemma:stability} $(i)$ and by the choice of prior \eqref{eq:prior-def} for $\alpha>3/2$.\\
    $(ii)$ This fact is proven in Appendix \ref{sec:cons-analytic}, since we deviate slightly from the setting of Theorem 1.3.2 in \cite{nickl2023}. 
\end{proof}
\begin{remark}
    Note the continuity of $\theta\mapsto m_\beta + e^\theta$ leaves us with convergence in probability on the level of $\beta$. However, the lack of uniform continuity  leaves us without a rate. A result with the above convergence rates on the level of $\beta$ is most easily achieved by replacing $\beta(\theta)=m_\beta+e^\theta$ with a smoothened `regular link function', in the sense of \cite{nickl2020b}. Our concern is only the well-definedness of the mean. Indeed, Bochner integrability as in \eqref{eq:bocnermean} seems not straightforward with $\theta$ replaced by $e^\theta$.
\end{remark}
This theorem justifies and quantifies the use of the Bayesian methodology for the two inverse problems. Note the theorem does not generalize immediately to the problem for Stokes' model with for example an $L^2$-norm on a set $K\subset\subset \Gamma_\beta$ in which $u\neq 0$ as in \cite[Remark 3.7]{boulakia2013b}. Indeed, the estimate includes the pressure $p$ and its normal derivative $\partial_\nu p|_\Gamma$ at $\Gamma$. Improving this estimate to be independent of observations of the pressure remains largely open to the authors knowledge.

\section{Experimental results}
\label{sec:experiments}
In this section, we illustrate the Bayesian methodology for both the Laplace problem \eqref{eq:laplace} for which the theoretical results where proven, and for the Stokes problem \eqref{eq:stokes} which motivated this study. 

\subsection{General methodology}
We consider a simple ground truth of the form 
$$\theta_0 = \sum_{k=-2}^2 \theta_{0,k} \phi_k$$ 
for $(\theta_{0,-2},\theta_{0,-1},\theta_{0,0},\theta_{0,1},\theta_{0,2})=(-0.6,  0.7,  2,  0.1, -0.08)$, and with $\phi_k(x)=\sin(2\pi kx)$ for $k>0$, $\phi_k(x)=\cos(2\pi kx)$ for $k<0$, and $\phi_0=1$. For simplicity, we truncate the prior series \eqref{eq:random-series} at $|k|=2$ for both $j=1$ and $j=2$. Furthermore, we choose $m_\beta=0$. As the computational domain, we consider the rectangle $\mathcal{O}=(0,1)\times (0,0.2)$ with $\Gamma = \Gamma_s = (0,1)\times \{1\}$, $\Gamma_0 = \{0\} \times (0,0.2) \cup \{1\} \times (0,0.2)$, and $\Gamma_\beta = (0,1)\times\{0\}$. For the Stokes' model \eqref{eq:stokes}, we add a homogeneous Neumann condition on $\Gamma_0$. The domain is represented by a triangular mesh consisting of $400\times 50$ elements. Forward computations are implemented in Python using FEniCSx \cite{AlnaesEtal2014}. The code is available at \cite{seizilles2023}.

\subsubsection{Synthetic data}
The data is generated following \eqref{eq:observations}, where we introduce  $\sigma_{\mathrm{noise}}$ to model the noise standard deviation. In the case of the Laplace model, the noisy observations can be seen in Figure \ref{fig:surfaceobservations} for the choice $h(x) = 10(\sin(12\pi x)+1)$ and $\sigma_{\mathrm{noise}} = 0.1$. Then the  likelihood function takes the form $\tilde \ell_N(\theta):=-\frac{1}{2\sigma^2_{\mathrm{noise}}}\sum_{i=1}^N |Y_i-\mathcal{G}(\theta)(X_i)|_V^2$.

\begin{figure}[t]
\centering
\begin{subfigure}{0.5\textwidth}
  \centering
  \includegraphics[width=1.\linewidth]{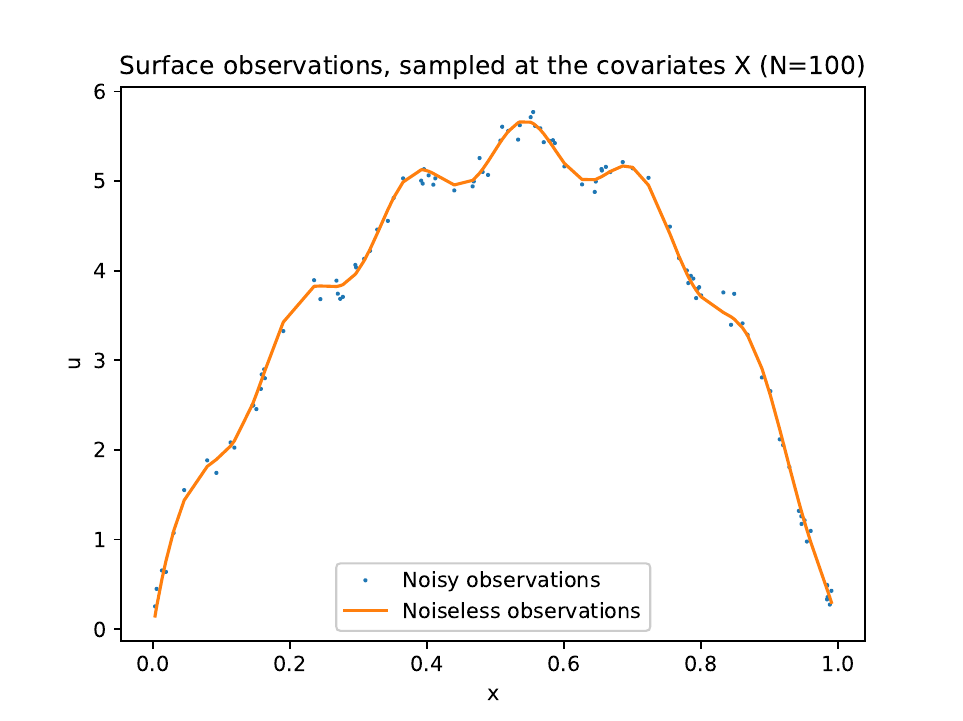}
\end{subfigure}%
\begin{subfigure}{0.5\textwidth}
  \centering
  \includegraphics[width=1.\linewidth]{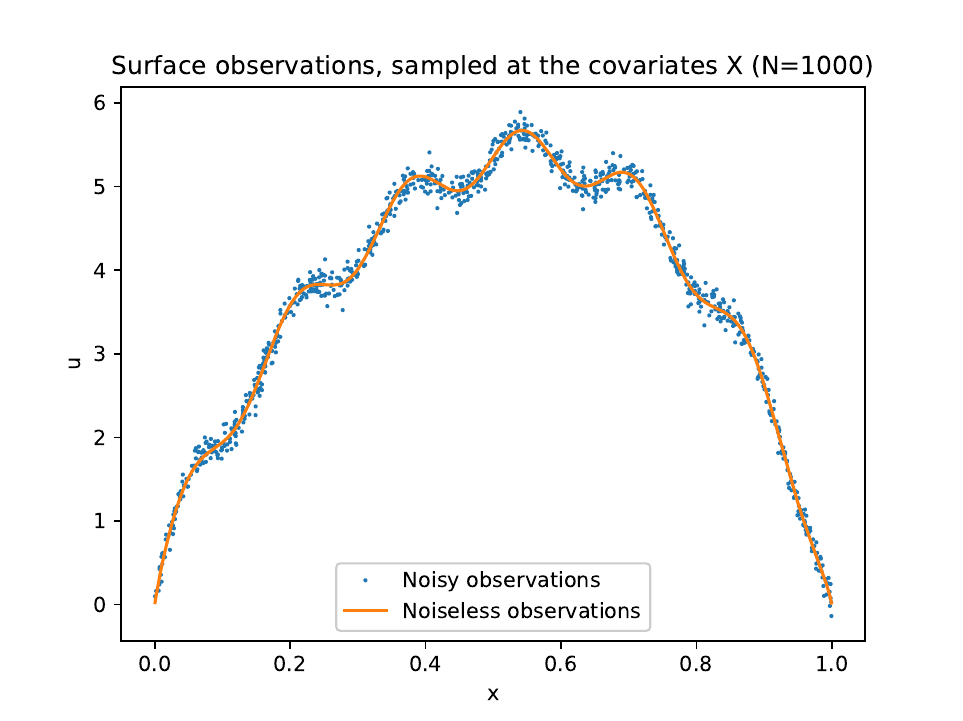}
\end{subfigure}

\caption{Surface observations for the Laplace problem. The orange line corresponds to the noiseless solution of the PDE parameterised with the true coefficients. After selecting covariates uniformly at random and adding independent Gaussian noise, one obtains the blue dots: noisy observations mimicking measurements and corresponding to the $(Y_i)_{i=1}^N$ in \ref{eq:observations}.}
\label{fig:surfaceobservations}
\end{figure}

\subsubsection{MCMC}\label{sec:mcmc}
We sample from the posterior distribution using the preconditioned Crank-Nicolson MCMC method, see \cite{hairer_spectral_2014}. To speed up the convergence, we employ methods of adaptive Monte-Carlo, see for example \cite{haario_adaptive_2001}. Here, we adapt the step size every $1000$ iterations to maintain an acceptance ratio close to a target rate of $0.33$ which seems to give reasonable sample diagnostics, starting at a step size of $\sqrt{2\gamma}$ with $\gamma=10^{-7}$. \\

For the Stokes model, the underlying PDE is more computationally expensive to solve. For this reason, we combine the adaptive step size (started this time with $\gamma=10^{-3}$) with a multilevel Monte Carlo approach, see \cite{christen_markov_2005}. In practice, two chains are run in parallel, one on a coarse mesh consisting of $150\times 30$ elements and the other on the fine mesh consisting of $400 \times 80$ elements. For every proposal, the pCN iteration is first performed on the coarse mesh, where the likelihood is less expensive to compute. If the proposal is accepted on the coarse mesh, then the pCN iteration is performed on the fine mesh using the same proposal. Monte Carlo estimates are performed using only samples of the chain at the fine level. 
In both problems, the chain is started close to the ground truth, with a shift drawn from $\mathcal{N}(0,0.5^2)$ for each dimension.

\subsection{Laplace problem} \label{sec:Laplaceexperiments}

Figures \ref{fig:matern_laplace} and \ref{fig:squared_exp_laplace} show the marginals of the posterior distribution, as well as the posterior mean estimate of the basal drag coefficient for a noise realization of noise level $\sigma_{\mathrm{noise}}=0.1$.
For both the Matérn and the squared exponential priors the uncertainty decreases as $N$ increases. In addition, the reconstruction from the posterior mean (i.e. $\beta(\hat{\theta})$, where $\hat{\theta}$ is the empirical posterior mean) visibly converges to the ground truth. We note that the ground truth is consistently contained in a $95\%$ credible interval.  One may notice a bigger uncertainty on the end points of the interval (at $x=0$ and $x=1$). This can be explained by the constraints on the model \ref{eq:laplace}, where a Dirichlet boundary condition forces the solution of the PDE system to be $u=0$ on the sides the domain, resulting in a ``loss of information" affecting the reconstruction. For $N=100$ observations, the performance of the Matérn prior and squared exponential prior are comparable. However, for $N=1000$, we observe a narrower credible interval. This effect is particularly visible for the coefficients corresponding to a ``large" frequency, which is expected for the squared exponential prior.\\

Figures of Markov chains (\ref{fig:chains}) for these experiments can be found in Appendix \ref{sec:morefigures}. As the number of observations $N$ increases the variance displayed for each chain decreases, matching what we observe when recovering the marginals.

\subsection{Stokes problem}\label{sec:stokes}

Going back to the motivating problem for this study, we perform the same simulations as in Sec.~\ref{sec:Laplaceexperiments}, this time with the Stokes PDE model \eqref{eq:stokes} instead of \eqref{eq:laplace}. We choose $h(x) = 10((\sin(12\pi x)+1),0)$ and pick $\rho=1$ and $g=(5,5)$. Here, we found $\sigma_{\mathrm{noise}}=0.5$ to be a suitable noise level. We consider only the squared exponential prior, since this prior shows the most promising theoretical and numerical results for the Laplace problem and this choice of ground truth. \\ 

Similarly as previously, as the number of observations $N$ increases the reconstruction from the posterior mean gets closer to the ground truth and the uncertainty is reduced. The marginals of the posterior distribution for these experiments can be found in \ref{fig:stokes_posterior}. Contrary to the experiments for the Laplace problem, at $N=100$ observations the ground truth is not consistently contained within the 95\% credible intervals, which is revised after increasing the number of observations to $N=1000$.

\begin{figure}[H]    

    \centering

    \includegraphics[width=13.5cm]{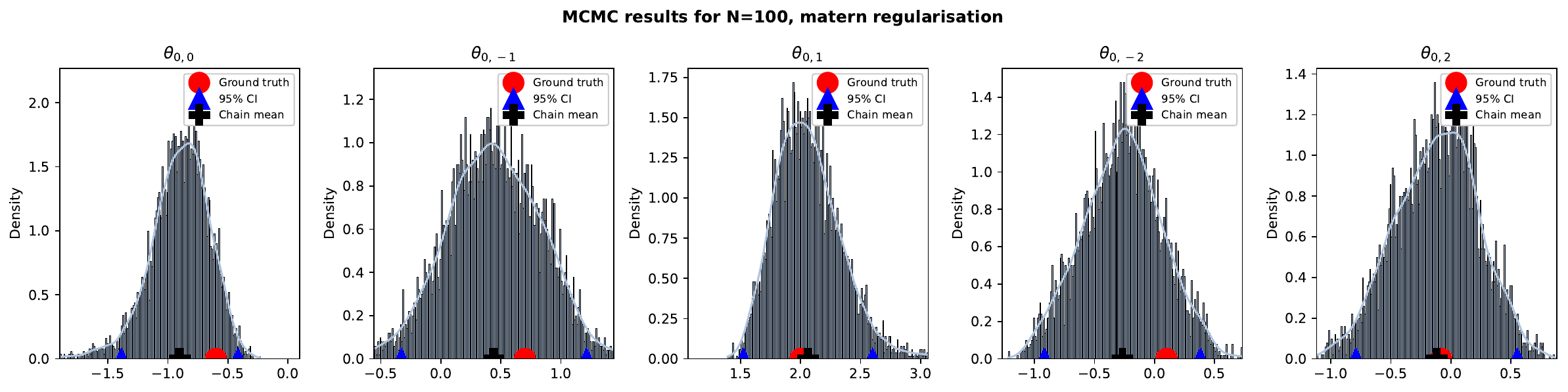}

    \includegraphics[width=13.5cm]{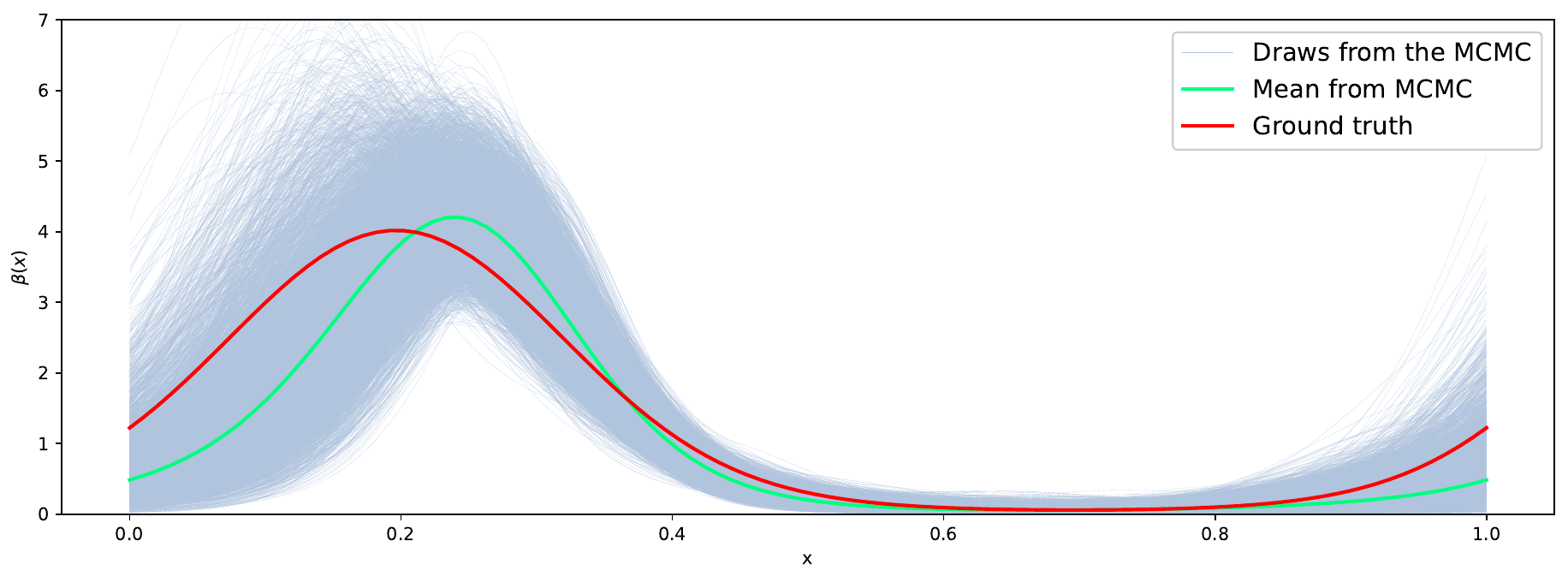}

    \includegraphics[width=13.5cm]{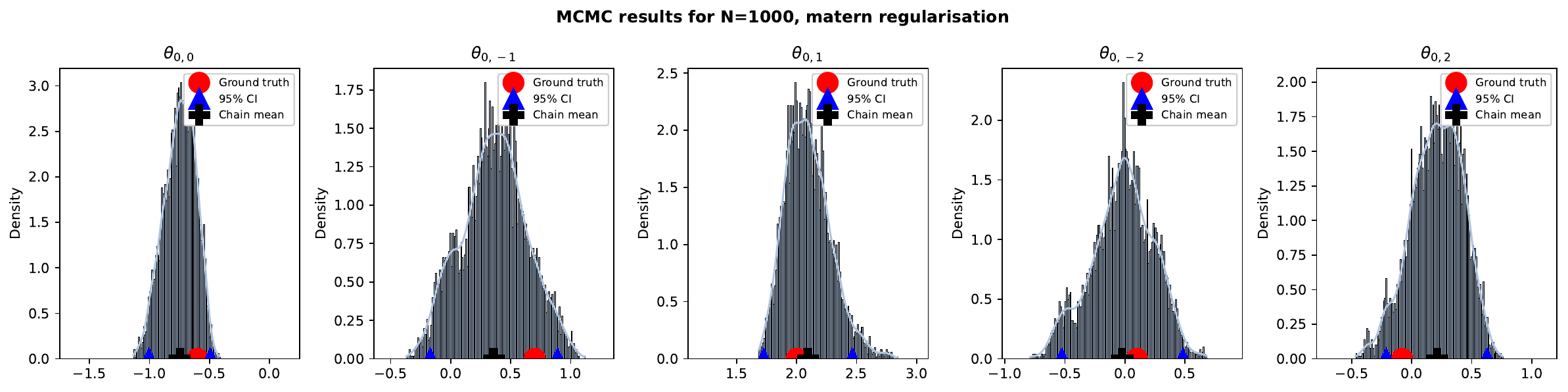}

    \includegraphics[width=13.5cm]{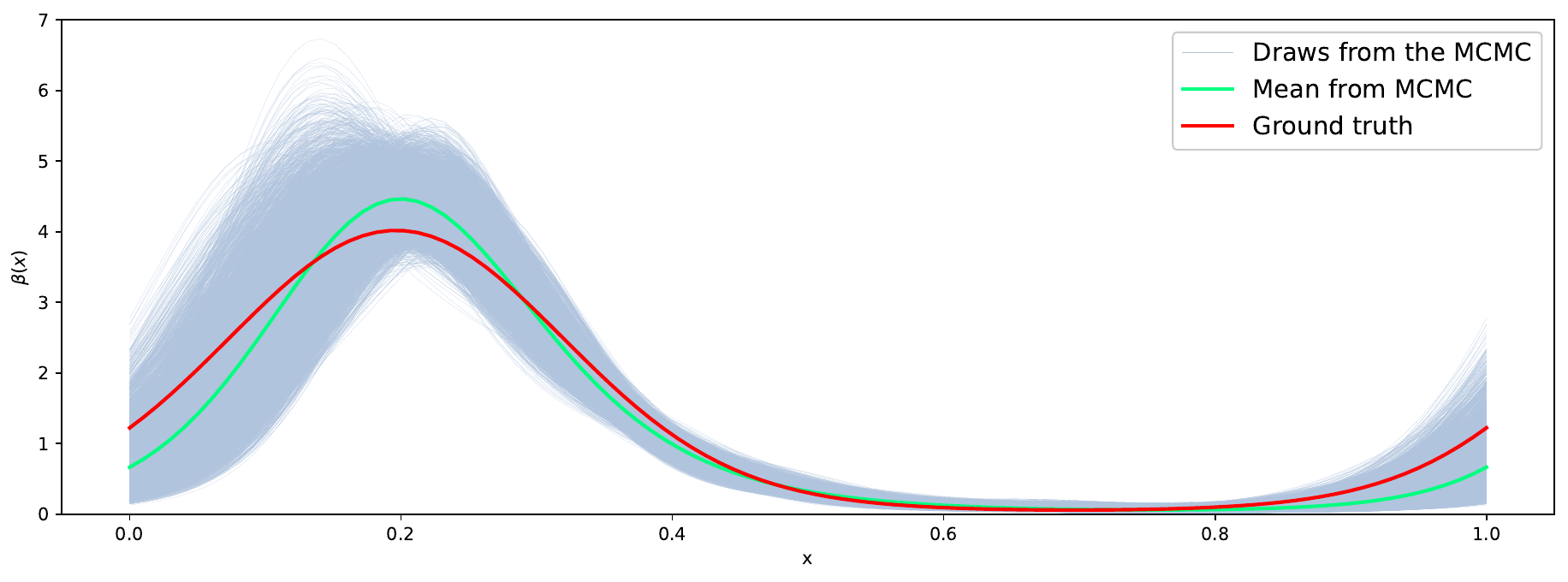}

    \caption{Histograms of $\theta$ and reconstruction of the drag coefficient $\beta$ for a Matérn-type prior ($\alpha=1$), with a number of observations of $N=100$ and $N= 1000$ respectively. For the histograms and the reconstructions, $5 000$ samples are taken equidistantly from the $500,000$ iterations of the chain after removal of the burn-in (first $100,000$ iterations).}%
    \label{fig:matern_laplace}
\end{figure}

\begin{figure}[H]
    \centering

    \includegraphics[width=13.5cm]{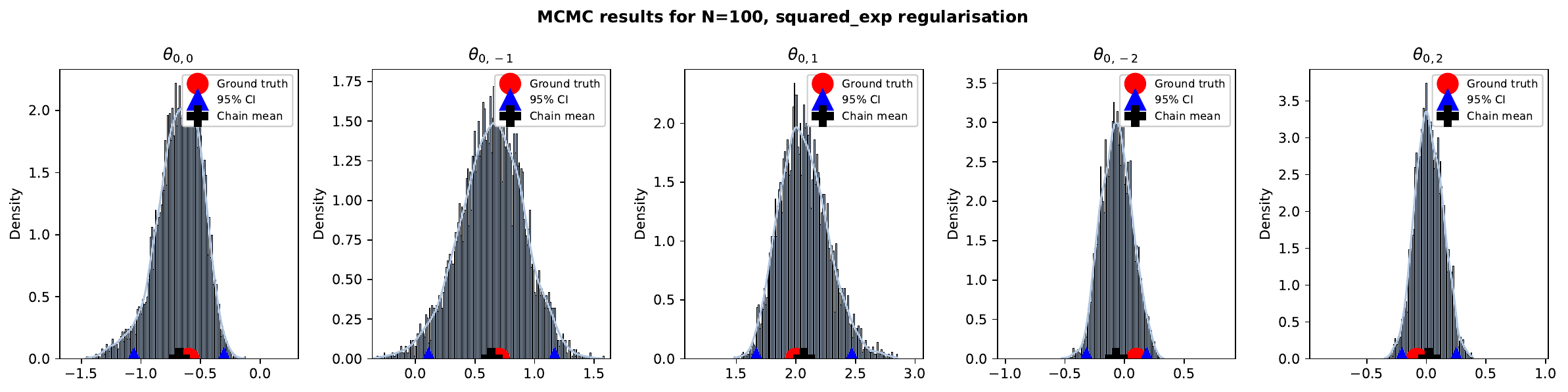}

    \includegraphics[width=13.5cm]{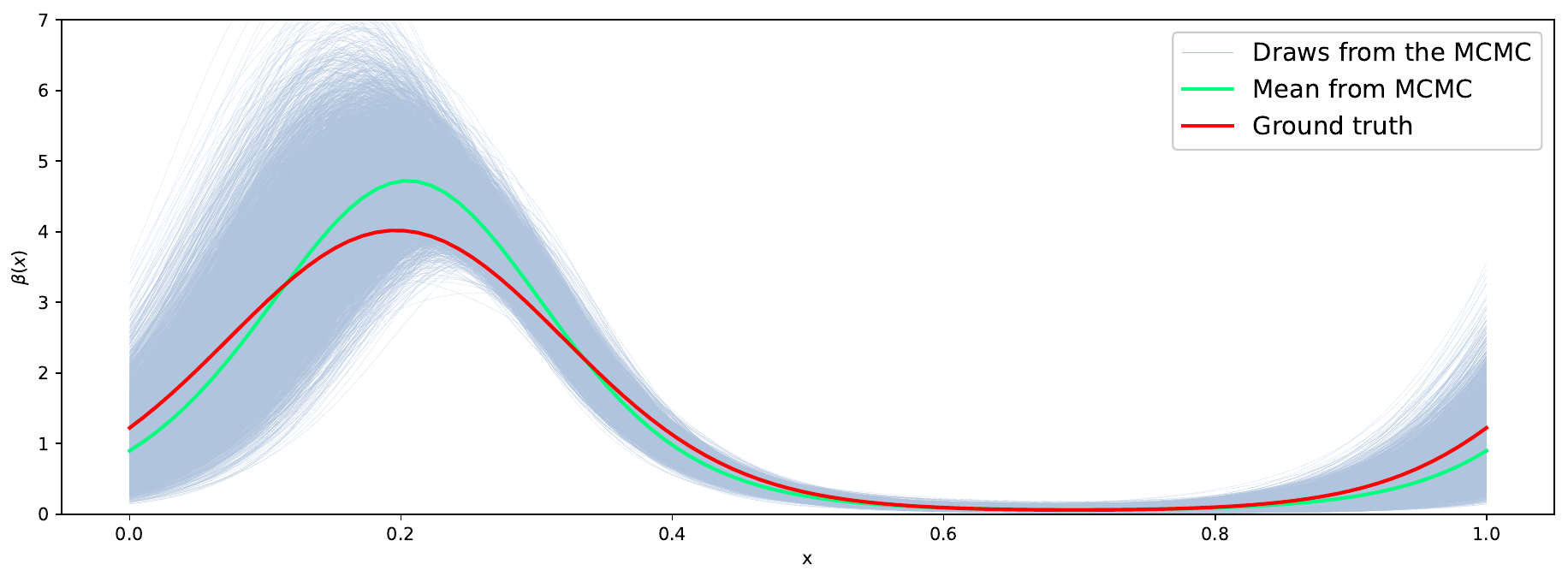}

    \includegraphics[width=13.5cm]{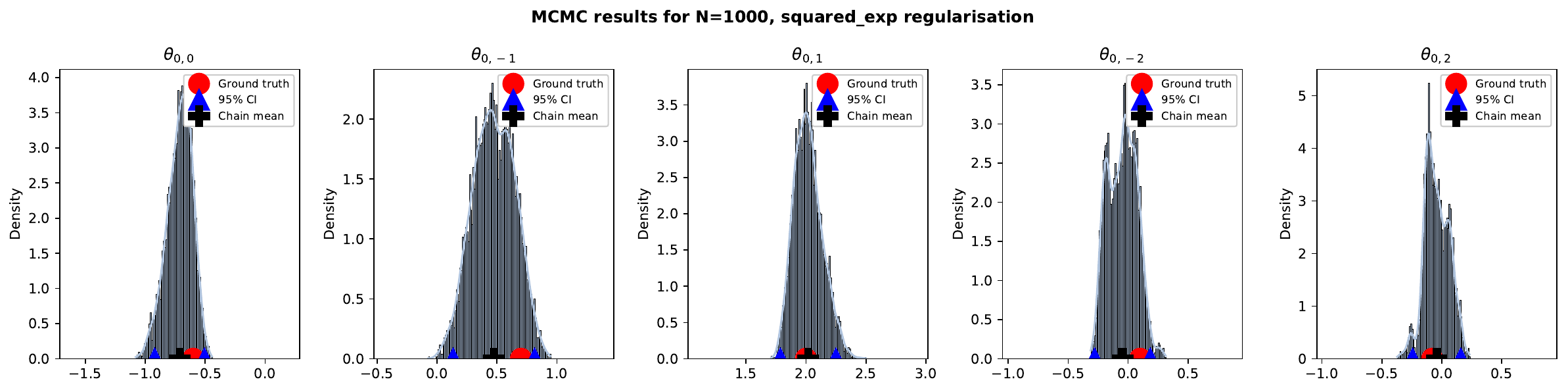}

    \includegraphics[width=13.5cm]{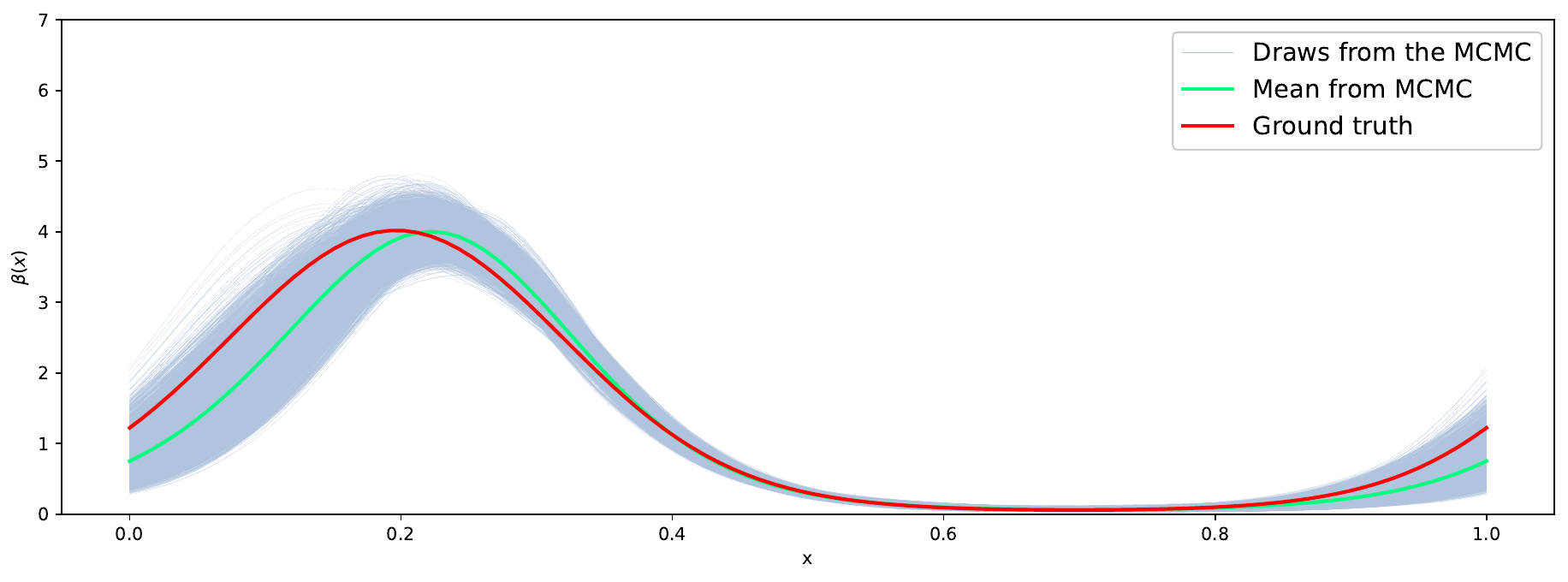}

    \caption{Histograms of $\theta$ and reconstruction of the drag coefficient $\beta$ for a squared exponential prior ($r=1$), with a number of observations of $N=100$ and $N= 1000$ respectively. For the histograms and the reconstructions, $5 000$ samples are taken equidistantly from the $500,000$ iterations of the chain after removal of the burn-in (first $100,000$ iterations).}
\label{fig:squared_exp_laplace}%
\end{figure}

\begin{figure}[H]
    \centering

    \includegraphics[width=13.5cm]{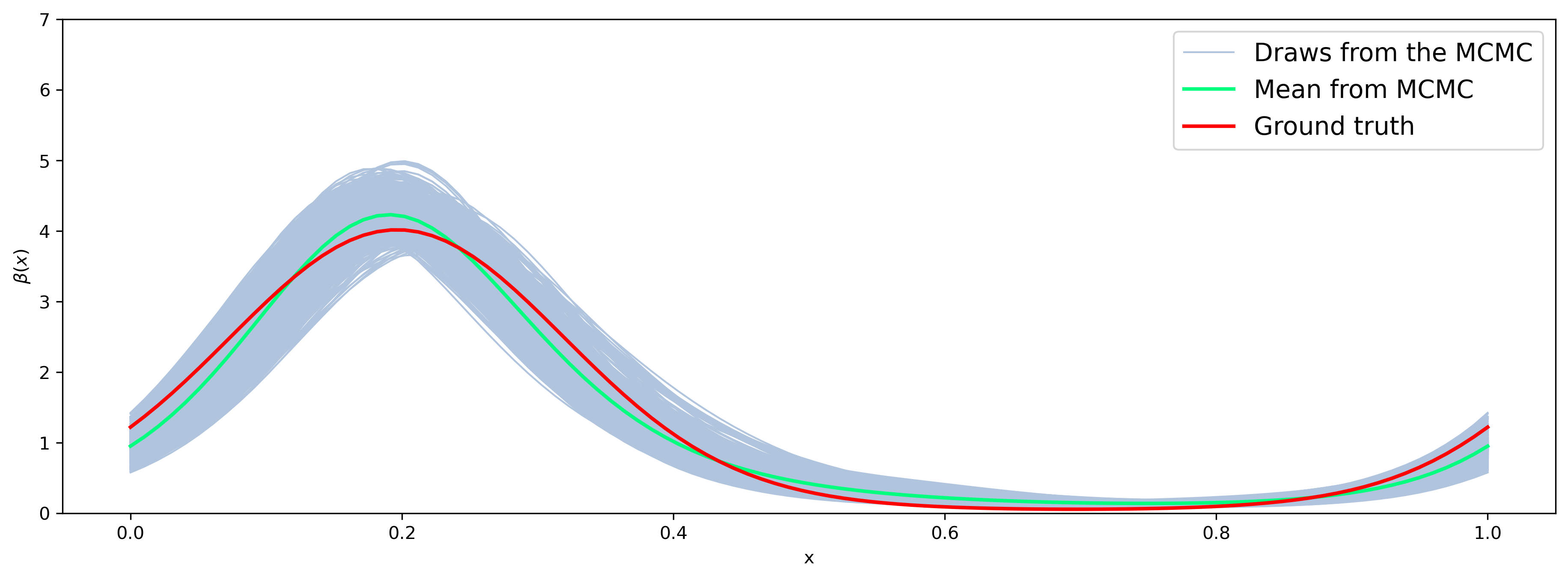}

    \includegraphics[width=13.5cm]{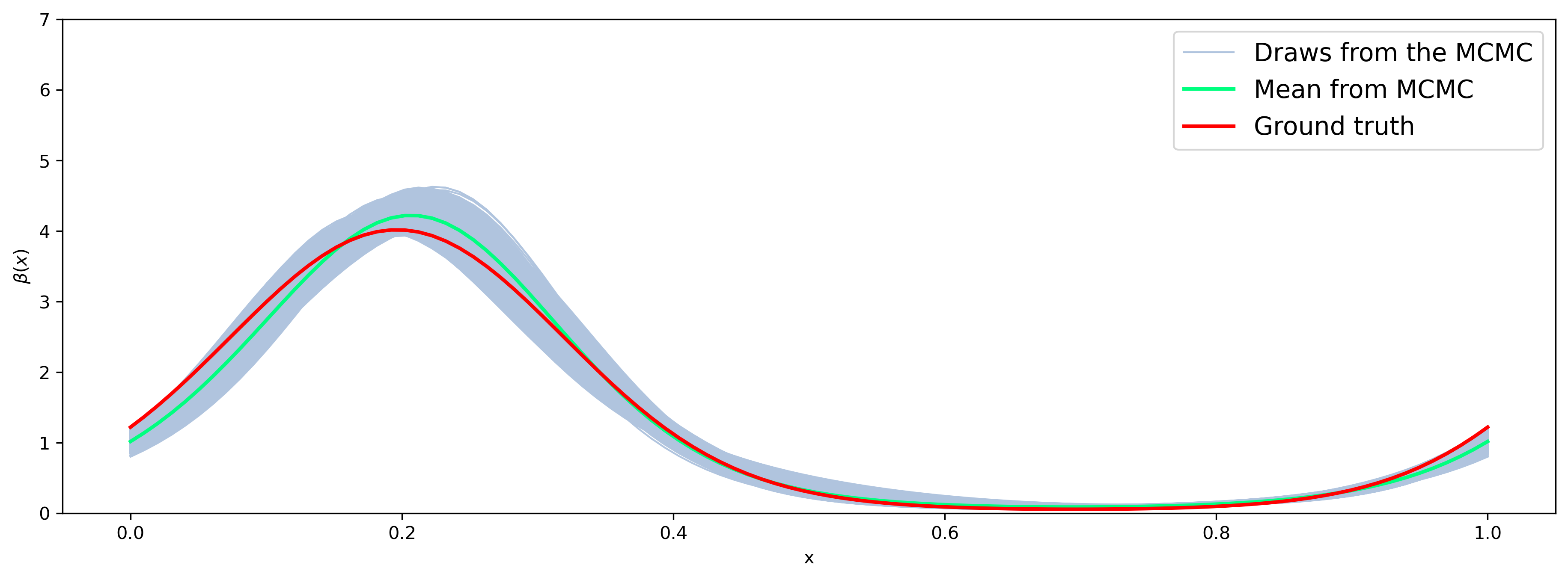}
    
    \caption{Reconstruction of the drag factor $\beta$ for a squared exponential prior ($r=1$), for the Stokes problem, with a number of observations of $N=100$ and $N= 1000$ respectively. For reconstruction, we take 1000 samples equidistantly from the fine chain after removal of the burn-in (first 1000 iterations, out of respectively 9000 and 7000 iterations for $N=100$ and $N=1000$). The coarse chains respectively consist of 44k and 61k iterations.}
    \label{fig:squared_exp_stokes}
\end{figure}

\section{Conclusions}
\label{sec:conclusions}
In this paper we considered a Bayesian approach to two inverse Robin problems with theoretical convergence guarantees as the number of observations increases. We have motivated to popular and numerically tractable Gaussian priors and show under appropriate rescaling that each lead to a convergent posterior mean. If the ground true Robin coefficient is \textit{a priori} known to be analytic, then the logarithmic convergence rate can be upgraded to a rate on the form $N^{-\tau}$ for some $\tau>0$. Interesting future work includes generalizing Theorem \ref{thm:main-theorem} to the inverse problem for Stokes' model. In its current form, Theorem \ref{thm:main-theorem} allows recovering analytic functions in the space $\mathcal{A}_r(\Gamma_\beta)$. Another interesting future direction is to generalize this to a larger class of analytic functions on $\Gamma_\beta$ using Gaussian processes and a continuous version of Lemma \ref{lemma:fourier-analytic}. For ideas in this direction we refer to \cite{vaart2009}.
Numerical experiments empirically confirmed that the reconstructions of the Robin coefficient improve as the number of observations increases. The main difficulty in the computations stems from the number of iterations required for the MCMC, since the likelihood needs to be evaluated at every step which requires solving a Laplace or Stokes PDE for every proposal of the parameter vector. In this paper, the pCN scheme was chosen for its simplicity. 
To speed up convergence further, one may consider the use of gradient-based MCMC methods.

\section*{Acknowledgments}
The authors would like to thank Prof. Richard Nickl for helpful discussions on Bayesian nonparametrics, as well as Dr Robert Arthern and Dr Rosie Williams for providing the motivation for this paper and for their assistance in understanding the physical aspects of the problem. The authors would also like to thank Prof. Kim Knudsen for helpful discussions on mixed boundary value problems. IK was funded by a Biometrika Fellowship awarded by the Biometrika Trust and the UCL IMSS Fellowship. AKR was supported by The Villum Foundation (grant no. 25893). FS was supported by Cantab Capital Institute for Mathematics of Information. MG was supported by a Royal Academy of Engineering Research Chair, and EPSRC grants EP/R018413/2, EP/P020720/2, EP/R034710/1, EP/R004889/1, EP/T000414/1.

\appendix
\section{Forward regularity}\label{sec:forward-reg}
\begin{proof}[Proof of Lemma \ref{lemma:well-posedness-stokes}]
$(i)$ Consider the general Stokes' equation for $f\in (L^2(\mathcal{O}))^2$, $h\in (H^{-1/2}(\Gamma))^2$ and $\tilde{h}\in (H^{-1/2}(\Gamma_\beta))^2$, 
\begin{equation}\label{eq:stokes-general}
    \begin{aligned}
    -\Delta u + \nabla p &=f \qquad &&\text{ in $\mathcal{O}$},\\
    \nabla \cdot u &= 0 &&\text{ in $\mathcal{O}$},\\
    \partial_{\nu} u - p\nu &= h &&\text{ on $\Gamma_s$},\\
    \partial_{\nu} u - p\nu + \beta u &= \tilde{h} &&\text{ on $\Gamma_\beta$}.
\end{aligned}
\end{equation}
The corresponding variational form is
\begin{equation}\label{eq:var-form-stokes}
    \int_{\mathcal{O}} \nabla u : \nabla v + \int_{\Gamma_\beta} \beta u \cdot v = \int_{\mathcal{O}} f \cdot v + \langle h, v\rangle_{-\frac{1}{2},\frac{1}{2},\Gamma_s} + \langle \tilde{h}, v\rangle_{-\frac{1}{2},\frac{1}{2},\Gamma_\beta},
\end{equation}
where $\nabla u: \nabla v$ denotes the double dot product of the two matrices, $\langle \cdot, \cdot \rangle_{-\frac{1}{2},\frac{1}{2},\Gamma_s}$ denotes the $(H^{-1/2}(\Gamma_s))^2, (H^{-1/2}(\Gamma_s))^2$ dual pairing, and where $v\in V_s:=\{v\in (H^1(\mathcal{O}))^2: \nabla \cdot v = 0 \}$. By the generalized Poincaré inequality, see for example \cite[Proposition 5.3.4]{brenner2008}, 
    $$\int_{\mathcal{O}} |\nabla u_i|^2\, + \int_{\Gamma_\beta} u_i^2\,\geq C(\mathcal{O}) \|u_i\|^2_{L^2(\mathcal{O})},$$
for each $i=1,2$, where $u_i$ is the $i'th$ component of the vector field $u$. It follows that
\begin{align}
    \int_{\mathcal{O}} \nabla u : \nabla u + \int_{\Gamma_\beta} \beta u \cdot u \geq C(m_\beta,\mathcal{O}) \sum_{k=1}^2 \|u_i\|_{H^1(\mathcal{O})}^2,
\end{align}
and hence the bilinear form is coercive. It is straightforward to check that it is also bounded, and likewise that the right-hand side is a bounded linear functional on $V_s$. By standard Lax-Milgram theory, there is a unique weak solution $u\in V_s$ to \eqref{eq:stokes-general} satisfying
\begin{equation}\label{eq:the-stokes-estimate}
    \|u\|_{(H^1(\mathcal{O}))^2} \leq C(\mathcal{O},m_\beta)(\|f\|_{(L^2(\mathcal{O}))^2} + \|h\|_{(H^{-1/2}(\Gamma_s))^2} +  \|\tilde{h}\|_{(H^{-1/2}(\Gamma_\beta))^2}).
\end{equation}
Note \eqref{eq:stokes-general} is in the form that Theorem IV.7.1 in \cite{boyer2013} considers with the compatibility condition being \eqref{eq:var-form-stokes} for $v=1$. Then there is also a unique solution $p\in L^2(\mathcal{O})$ to \eqref{eq:stokes-general}. In the following we take some care in bounding this function. Initially de Rhams' theorem \cite[Theorem IV2.4]{boyer2013} gives a pressure term $\tilde{p}\in L^2_0(\mathcal{O})=L^2(\mathcal{O})/\mathbb{R}$ unique up to a constant and satisfying $-\Delta u+\nabla \tilde{p}= f$. Take then the mean-zero solution satisfying
\begin{align}
    \|\tilde{p}\|_{L^2(\mathcal{O})}&\leq C(\mathcal{O})\|\nabla \tilde{p}\|_{H^{-1}(\mathcal{O})},\\
    &= C(\mathcal{O})\|\Delta u\|_{H^{-1}(\mathcal{O})} + \|f\|_{H^{-1}(\mathcal{O})},\\
    &\leq C(\mathcal{O},m_\beta)(\|f\|_{(L^2(\mathcal{O}))^2} + \|h\|_{(H^{-1/2}(\Gamma_s))^2} +  \|\tilde{h}\|_{(H^{-1/2}(\Gamma_\beta))^2}) \label{eq:bound-on-tilde}
\end{align}
using \cite[Lemma IV.1.9]{boyer2013} and \eqref{eq:the-stokes-estimate}. The proof of Theorem IV.7.1 in \cite{boyer2013} shows that $p=\tilde{p}+C_0$ is the unique solution to \eqref{eq:stokes-general} matching the boundary conditions. If $h\in (L^2(\Gamma_s))^2$ and $\tilde{h}\in (L^2(\Gamma_\beta))^2$, this constant can be bounded as
\begin{align}
    |C_0|&\leq C(\mathcal{O})(\|h\|_{(L^2(\Gamma_s))^2} +  \|\tilde{h}\|_{(L^2(\Gamma_\beta))^2} + \|\beta\|_{L^\infty(\Gamma_\beta)}\|u\|_{L^2(\mathcal{O})}\\
    &\,\,\,\,\,+ \|\partial_\nu u\|_{H^{-1/2}(\partial\mathcal{O})} + \|\tilde{p}\|_{L^2(\mathcal{O})}),
\end{align}
hence if $\|\beta\|_{L^\infty(\Gamma_\beta)}\leq M$, then
\begin{equation}\label{eq:bound-on-p}
    \|p\|_{L^2(\mathcal{O})}\leq C(\mathcal{O},m_\beta,M)(\|f\|_{(L^2(\mathcal{O}))^2} + \|h\|_{(L^2(\Gamma_s))^2} +  \|\tilde{h}\|_{(L^2(\Gamma_\beta))^2}) 
\end{equation}
using \eqref{eq:bound-on-tilde}.\\
$(ii)$ The difference $(v,q)$ for $v=u_1-u_2$ and $q=p_1-p_2$ of solutions $(u_1,p_1),(u_2,p_2)$ of \eqref{eq:stokes} corresponding to $\beta_1=\beta(\theta_1),\beta_2=\beta(\theta_2)$ is the unique solution of
\begin{equation}\label{eq:stokes-general-diff}
    \begin{aligned}
    -\Delta v + \nabla q &= 0 \qquad &&\text{ in $\mathcal{O}$},\\
    \nabla \cdot v &= 0 &&\text{ in $\mathcal{O}$},\\
    \partial_{\nu} v - q\nu &= 0 &&\text{ on $\Gamma_s$},\\
    \partial_{\nu} v - q\nu + \beta_1 v &= u_2(\beta_2-\beta_1) &&\text{ on $\Gamma_\beta$}.
\end{aligned}
\end{equation}
Note that \cite[Lemma 2.3]{boulakia2013a} implies
$$\|u_2(\beta_2-\beta_1)\|_{(L^2(\Gamma_\beta))^2} \leq \|u_2\|_{(L^2(\Gamma_\beta))^2}\|\beta_2-\beta_1\|_{H^{1}(\Gamma_\beta)},$$
and hence by $(i)$ above,
\begin{align}
    \|v\|_{(H^1(\mathcal{O}))^2} &\leq C(\mathcal{O},m_\beta)\|u_2(\beta_2-\beta_1)\|_{(H^{-1/2}(\Gamma_\beta))^2},\\
    &\leq C(\mathcal{O},m_\beta,h,\rho,g)\|\beta_1-\beta_2\|_{H^{1}(\Gamma_\beta)}.
\end{align}
To upgrade this we prove additional smoothness of $v$ near $\Gamma$ as follows. Define an open set $V\subset \mathcal{O}$ that meets $\Gamma$, i.e. $\Gamma\subset \overline{V}$. Define then the larger set $U\subset \mathcal{O}$ with $V\subset U$ and $\overline{U}\cap \partial\mathcal{O}\subset \Gamma_s$. We then define the smooth cutoff function $\eta\in C^\infty(U)$ with $\eta \equiv 1$ in $V$ and $\supp \eta \subset \overline{U}$ (hence $\eta$ is zero near $\Gamma_0$ and $\Gamma_\beta$). Then $(\eta v, \eta q)$ solves the system
\begin{equation}\label{eq:stokes-general-diff-cutoff}
    \begin{aligned}
    -\Delta (\eta v) + \nabla (\eta q) &= \tilde{f} \qquad &&\text{ in $U$},\\
    \nabla \cdot (\eta v) &= \nabla \eta \cdot v &&\text{ in $U$},\\
    \partial_{\nu} (\eta v) - (\eta q)\nu &= \eta(\partial_\nu v - q\nu) + v\partial_\nu \eta &&\text{ on $\partial U$},
\end{aligned}
\end{equation}
for $\tilde{f}=v\Delta \eta + 2\nabla v \cdot \nabla \eta + q\nabla\eta \in L^2(\mathcal{O})$. Note $\eta(\partial_\nu v - q\nu) + v\partial_\nu \eta = v\partial_\nu \eta \in (H^{1/2}(\partial U))^2$. Then Theorem IV.7.1 of \cite{boyer2013} states that
$$\|\eta v\|_{(H^2(U))^2}\leq C(U)(\|\tilde{f}\|_{(L^2(U))^2}+\|\nabla \eta \cdot v\|_{H^1(U)}+\|v\partial_\nu \eta\|_{(H^{1/2}(\partial U))^2}).$$
Since $\eta \equiv 1$ in $V$ and using \eqref{eq:bound-on-p}, we have
$$\|v\|_{(H^2(V))^2}\leq C(\mathcal{O},m_\beta,M,h,\rho,g)$$
By Sobolev interpolation, there exists $\alpha,\tilde{\alpha}>0$ such that (denoting $v_i$ the $i'th$ component of $v$)
\begin{align}\label{eq:v-to-beta}
    \sum_{i=1}^2 \|v_i\|_{(H^{7/4}(V))^2} &\leq \sum_{i=1}^2 \|v_i\|_{(H^1(V))^2}^\alpha \|v_i\|_{(H^2(V))^2}^{\tilde{\alpha}},\\
    &\leq C(\mathcal{O},m_\beta,h,\rho,g) \sum_{i=1}^2 \|v_i\|_{(H^1(V))^2}^\alpha,\\
    &\leq C(\mathcal{O},m_\beta,h,\rho,g) \|\beta_1-\beta_2\|_{H^{1}(\Gamma_\beta)}^\alpha.
\end{align}
Now we argue that $\theta\mapsto e^\theta$ is locally Lipschitz continuous in $H^1(\Gamma_\beta)$, i.e.
\begin{equation}\label{eq:beta-to-theta}
    \|\beta_1-\beta_2\|_{H^{1}(\Gamma_\beta)} \leq C(\Gamma_\beta,M)\|\theta_1-\theta_2\|_{H^{1}(\Gamma_\beta)}.
\end{equation}
Note first by the mean value theorem that
$$\|e^{\theta_1}-e^{\theta_2}\|_{L^\infty(\Gamma_\beta)}\leq e^M \|\theta_1-\theta_2\|_{L^\infty(\Gamma_\beta)}.$$
If $\Gamma_\beta \subset \mathbb{R}$ it is clear that
\begin{align}
    \|\beta_1-\beta_2\|_{H^1(\Gamma_\beta)} &\lesssim \|e^{\theta_1}-e^{\theta_2}\|_{L^2(\Gamma_\beta)} + \|\nabla(e^{\theta_1})-\nabla(e^{\theta_2})\|_{L^2(\Gamma_\beta)},\\
    &\leq \|e^{\theta_1}-e^{\theta_2}\|_{L^2(\Gamma_\beta)} + \|e^{\theta_1}\nabla\theta_1-e^{\theta_2}\nabla\theta_2\|_{L^2(\Gamma_\beta)},\\
    &\leq \|e^{\theta_1}-e^{\theta_2}\|_{L^2(\Gamma_\beta)} + \|\nabla\theta_1\|_{L^2(\Gamma_\beta)}\|e^{\theta_1}-e^{\theta_2}\|_{L^\infty(\Gamma_\beta)}\\
    &\,\,\,\,\,\,+\|e^{\theta_2}\|_{L^\infty(\Gamma_\beta)}\|\nabla\theta_1-\nabla\theta_2\|_{L^2(\Gamma_\beta)},\\
    &\leq C(M) \|\theta_1-\theta_2\|_{H^1(\Gamma_\beta)},
\end{align}
using also the continuous Sobolev embedding $H^1(\Gamma_\beta)\subset C(\overline{\Gamma}_\beta)$. 
By the definition of Sobolev spaces on boundaries, see \cite[(1,3,3,2)]{grisvard1985}, the case where $\Gamma_\beta$ is a smooth curve follows in the same way. Indeed, this amounts to showing
$$\|\beta_1\circ \phi -\beta_2 \circ \phi\|_{H^1(I)} \leq C \|\theta_1\circ \phi-\theta_2\circ\phi\|_{H^1(I)}$$
for any smooth parametrization $\phi:I\rightarrow \mathbb{R}^2$ of a section of $\Gamma_\beta$ with $I$ an open subset of $\mathbb{R}$. In this case we just repeat the argument above.
Finally, combining \eqref{eq:v-to-beta} with \eqref{eq:beta-to-theta} and a Sobolev embedding it follows that
$$\|u_1-u_2\|_{(C(\overline{\Gamma}))^2}\lesssim \sum_{i=1}^2 \|v_i\|_{C(\overline{\Gamma})} \leq C(\mathcal{O},m_\beta,h,\rho,g,M) \|\theta_1-\theta_2\|_{H^{1}(\Gamma_\beta)}^\alpha.$$
$(iii)$ This is proved in Proposition 3.3 of \cite{boulakia2013a} for a stationary Neumann condition $g(x,t)=h(x)$ in $(H^{1/2}(\Gamma_s))^2$.

\end{proof}

\begin{proof}[Proof of Lemma \ref{lemma:classicH1}]
    Consider more generally the equation \eqref{eq:laplace} for an inhomogeneous Robin condition $\partial_\nu u+\beta u = \tilde{h}\in H^{1/2}(\Gamma_\beta)$. The corresponding variational form is 
    \begin{equation}\label{eq:varform}
        \int_{\mathcal{O}} \nabla u \cdot \nabla v \, + \int_{\Gamma_\beta} \beta u v\,  = \int_{\Gamma} hv\, + \int_{\Gamma_\beta} \tilde{h}v,
    \end{equation}
    for $v\in V:=\{u\in H^1(\mathcal{O}): u|_{\Gamma_0}=0\}$. By the generalized Poincaré inequality, see for example \cite[Proposition 5.3.4]{brenner2008}, 
    $$\int_{\mathcal{O}} |\nabla u|^2\, + \int_{\Gamma_\beta} u^2\,\geq C(\mathcal{O}) \|u\|^2_{L^2(\mathcal{O})},$$
    hence the left-hand side of \eqref{eq:varform} is a coercive bilinear form on $V$. Since $h$ and $\tilde{h}$ are $H^{1/2}$-functions, the right-hand side is a bounded linear functional on $V$. By standard Lax-Milgram theory, there is a unique weak solution $u\in V$ to \eqref{eq:varform} satisfying
    \begin{equation}\label{eq:u-estimate}
        \|u\|_{H^1(\mathcal{O})} \leq C(\mathcal{O},m_\beta) (\|h\|_{H^{-1/2}(\Gamma)}+\|\tilde{h}\|_{H^{-1/2}(\Gamma_\beta)}).
    \end{equation}
    In particular, \eqref{eq:u-estH1} is satisfied.
\end{proof}

\begin{lemma}\label{lemma:higher-reg}
    For $\beta\in H^1(\Gamma_\beta)$ with $\|\beta\|_{H^1(\Gamma_\beta)}\leq M$, $h$ as in Assumption \ref{assump:h}, and any $0<s<\frac{1}{2}$ there exists a constant $C=C(\mathcal{O},m_\beta,M,M_h,s)$ such that
    \begin{equation}
        \|u\|_{H^{1+s}(\mathcal{O})} \leq C,
    \end{equation}
    where $u$ solves \eqref{eq:laplace}.
\end{lemma}
\begin{proof}
Far away from the `corners' (where different boundary conditions meet) the estimate is straightforward using standard techniques. Near the corners the estimate is essentially due to \cite{grisvard1985}, although we are aided by \cite{banasiak1989}. Since $\beta u\in H^{1/2}(\Gamma_\beta)$ by Lemma 2.3 in \cite{boulakia2013a} ($u\in H^{1/2}(\Gamma_\beta)$ by Lemma \ref{lemma:classicH1} and $\beta \in H^1(\Gamma_\beta)$ by assumption), the trace theorem in \cite[Theorem 2.1]{banasiak1989} provides a function $v\in H^2(\mathcal{O})$ such that $\partial_\nu v =h$ on $\Gamma$, $v = 0$ on $\Gamma_0$ and $\partial_\nu v = -\beta u$ on $\Gamma_\beta$, i.e. $w=u-v$ solves
\begin{equation}\label{eq:homogeneous-pde}
    \begin{aligned}
    \Delta w &= -\Delta v \qquad &&\text{ in $\mathcal{O}$},\\
    \partial_{\nu} w &= 0 &&\text{ on $\Gamma$},\\
    w &= 0 && \text{ on $\Gamma_0$},\\
    \partial_{\nu} w &= 0 &&\text{ on $\Gamma_\beta$}.
\end{aligned}
\end{equation}
Indeed this trace operator $T:H^2(\mathcal{O})\rightarrow H^{1/2}(\Gamma)\times H^{3/2}(\Gamma_0)\times H^{1/2}(\Gamma_\beta)$, defined by $u\mapsto (\partial_\nu u|_\Gamma, u|_{\Gamma_0},\partial_\nu u|_{\Gamma_\beta})$, is bounded \cite{lions1972} and surjective \cite[Theorem 2.1]{banasiak1989}, so there exists a continuous right-inverse, see the general remark after Theorem 8.3 in \cite{lions1972}. Then
\begin{align}\label{eq:trace-est}
    \|v\|_{H^2(\mathcal{O})}&\leq C(\|h\|_{H^{1/2}(\Gamma)}+\|\beta u\|_{H^{1/2}(\Gamma_\beta)}),\\
    &\leq C(\|h\|_{H^{1/2}(\Gamma)}+\|\beta\|_{H^1(\Gamma_\beta)} \|u\|_{H^{1/2}(\Gamma_\beta)}),
\end{align}
The regularity decomposition of \cite[Theorem 3.11]{banasiak1989} decomposes the unique solution $w\in H^1(\mathcal{O})$ as
$$w = w_r + \sum_{j=1}^J c_j S_j,$$
where $w_r\in D^2:=\{w\in H^2(\mathcal{O}): \partial_{\nu}w= 0 \text{ on } \Gamma\cup \Gamma_\beta, w=0 \text{ on } \Gamma_0\}$, $c_j=c_j(f)$ are functionals of $f=-\Delta v$ in $L^2(\mathcal{O})$, see \cite[Remark 3.1.2]{banasiak1989} and $S_j$ are certain `singular' function supported near the corners. They depend only on the geometry of $\mathcal{O}$, see (3.2.26) and Proposition 3.2.3 in \cite{banasiak1989}, and satisfy $\Delta S_j\in L^2(\mathcal{O})$ and $S_j\in H^{1+s}(\mathcal{O})$ if and only if $s<1/2$. 
Since $\Delta:D^2\rightarrow L^2$ is injective by uniqueness of solutions to \eqref{eq:homogeneous-pde}, it is bijective onto its image. The open mapping theorem then states that there exists a constant $C>0$ such that $\|w_r\|_{H^2(\mathcal{O})}\leq C\|\Delta w_r\|_{L^2(\mathcal{O})}$, and hence
\begin{align}
    \|w_r\|_{H^2(\mathcal{O})}&\leq C \|\Delta w_r\|_{L^2(\mathcal{O})},\\
    &\leq C(\|\Delta v\|_{L^2(\mathcal{O})}+\sum_{j=1}^J|c_j|\|\Delta S_j\|_{L^2(\mathcal{O})}),\\
    &\leq C(\mathcal{O}) \|v\|_{H^2(\mathcal{O})}. \label{eq:estimate}
\end{align}
Combining \eqref{eq:estimate} with \eqref{eq:trace-est} and using the standard estimate of $\|u\|_{H^1(\mathcal{O})}$ we have
\begin{align}
    \|u\|_{H^{1+s}(\mathcal{O})} &\leq C(\mathcal{O})\|v\|_{H^2(\mathcal{O})} + \|\sum_{j=1}^J c_j S_j\|_{H^{1+s}(\mathcal{O})},\\
    &\leq C(\mathcal{O},s)\|v\|_{H^2(\mathcal{O})}\leq C.
\end{align}
with $C=C(\mathcal{O},M_h,M,m_\beta).$
\end{proof}

\begin{proof}[Proof of Lemma \ref{lemma:forward-reg}]
    $(i)$ is an immediate consequence of a Sobolev embedding and Lemma \ref{lemma:higher-reg}.\\
    $(ii)$ The difference $v=u_1-u_2$ of solutions $u_1, u_2$ corresponding to $\beta_1=\beta(\theta_1), \beta_2=\beta(\theta_2)$ is the unique solution to the equation 
    \begin{equation}\label{eq:diff}
    \begin{aligned}
    \Delta v &= 0 \qquad &&\text{ in $\mathcal{O}$},\\
    \partial_{\nu} v &= 0 &&\text{ on $\Gamma$},\\
    v &= 0 && \text{ on $\Gamma_0$},\\
    \partial_{\nu} v + \beta_1 v &= u_2(\beta_2-\beta_1) &&\text{ on $\Gamma_\beta$}.
\end{aligned}
\end{equation}
Since $u_2(\beta_2-\beta_1)\in H^{-1/2}(\Gamma_\beta)$, we use the estimate \eqref{eq:u-estimate} with $h=0$ and $\tilde{h}=u_2(\beta_2-\beta_1)$ to the effect that
\begin{align}\label{eq:forward-reg}
        \|v\|_{H^1(\mathcal{O})} &\leq C(\mathcal{O},m_\beta) \|u_2(\beta_2-\beta_1)\|_{H^{-1/2}(\Gamma_\beta)},\\
        &\leq C(\mathcal{O},m_\beta) \|\beta_1-\beta_2\|_{L^\infty(\Gamma_\beta)},\\
        &\leq C(\mathcal{O},m_\beta,M) \|\theta_1-\theta_2\|_{L^\infty(\Gamma_\beta)},
\end{align}
using a simple mean value theorem argument. Boundedness of the trace operator implies $(ii)$.\\
$(iii)$ By Sobolev interpolation, there exists $\alpha,\tilde{\alpha}>0$ such that
\begin{align}
    \|v\|_{H^{1+1/8}(\mathcal{O})} &\leq \|u_1-u_2\|_{H^1(\mathcal{O})}^\alpha\|u_1-u_2\|_{H^{1+1/4}(\mathcal{O})}^{\tilde{\alpha}},\\
    &\leq C(\mathcal{O},m_\beta,M,M_h) \|\theta_1-\theta_2\|_{L^\infty(\Gamma_\beta)}^\alpha,
\end{align}
where we used \eqref{eq:forward-reg} and Lemma \ref{lemma:higher-reg}. Then boundedness of the trace operator and a Sobolev embedding give the wanted result.
\end{proof}

\section{Conditional stability estimates}\label{sec:cond-stab-est}
\begin{proof}[Proof of Lemma \ref{lemma:stability}]
Notice first that the mean value theorem for $\tilde{\theta}(x)\in [\theta_1(x),\theta_2(x)],$
$$\beta_1-\beta_2=e^{\theta_1}-e^{\theta_2}=e^{\tilde{\theta}}(\theta_1-\theta_2),$$
implies
$$\|\theta_1-\theta_2\|_{L^q(\Gamma_\beta)}\leq C(M)\|\beta_1-\beta_2\|_{L^q(\Gamma_\beta)},$$
for any $1\leq q \leq \infty$, since $\|\tilde{\theta}\|_{L^\infty(\Gamma_\beta)}\leq M$ in either case of our assumptions. It is then sufficient to consider stability estimate on the level of $\beta$. \\
$(i)$ Theorem 2.2 of \cite{alessandrini2003} states that
\begin{equation}\label{eq:ale-estimate}
    \|\beta_1-\beta_2\|_{L^\infty(\Gamma_{\beta,\epsilon})}\leq \tilde{K} | \log (\|\mathcal{G}(\theta_1)-\mathcal{G}(\theta_2)\|_{L^\infty(\Gamma)}) |^{-\sigma}
\end{equation}
for some $\tilde{K}>0$ and $0<\sigma<1$ dependent on $\mathcal{O}$, $h$, $M_1$ and $\epsilon$. Sobolev embedding and interpolation results gives for some $0<\delta <\frac{1}{4}$
\begin{align}
    \|\mathcal{G}(\theta_1)-\mathcal{G}(\theta_2)\|_{L^\infty(\Gamma)} &\leq \|\mathcal{G}(\theta_1)-\mathcal{G}(\theta_2)\|_{H^{\frac{1}{2}+\delta}(\Gamma)},\\
    &\leq \|\mathcal{G}(\theta_1)-\mathcal{G}(\theta_2)\|_{L^2(\Gamma)}^{p}\|\mathcal{G}(\theta_1)-\mathcal{G}(\theta_2)\|_{H^{\frac{1}{2}+2\delta}(\Gamma)}^{1-p},\\
    &\leq M(\mathcal{O},m_\beta,M_\beta,M_h,\delta) \|\mathcal{G}(\theta_1)-\mathcal{G}(\theta_2)\|_{L^2(\Gamma)}^{p},
\end{align}
where $p=\frac{2\delta}{1+4\delta}$, and where we used Lemma \ref{lemma:higher-reg}. Inserting this into \eqref{eq:ale-estimate} for some fixed $\delta$ gives $(i)$ for $K=K(\tilde{K},M)$.\\
$(ii)$ We follow the argument of \cite[Section 3]{hu2015}, which relies on two auxillary results: 
\begin{enumerate}
    \item $\min_{x\in \Gamma_{\beta,\epsilon}}u(x)\geq \eta$, where $\eta>0$ is a constant dependent on $\epsilon$, but independent of the imposed boundary condition on $\Gamma$.
    \item the solution $u$ to \eqref{eq:laplace} can be analytically extended in a fixed neighborhood $U$ of $\Gamma_\beta$ with $\|u\|_{H^2(U)}\leq C(M)$, where it is also harmonic.
\end{enumerate}
In the presence of these two results, the estimate follows exactly as in \cite[Theorem 3.1]{hu2015} with $K>0$ and $0<\sigma<1$ depending only on $M$, $\epsilon$, $\mathcal{O}$, $M_h$, $M$, and we will not repeat it here.\\\\
(1) Note first that $u(x)\geq \eta$ for any $x\in \Gamma_{\beta,\epsilon}$, where $\eta>0$ is some constant depending on $\epsilon$, but independent of $h$. This follows from continuity of $u$ on $\overline{\mathcal{O}}$ and maximum principles for harmonic functions as in \cite[Lemma 3.2]{hu2015}. Indeed, one can conclude that $u\geq 0$ everywhere on $\overline{\mathcal{O}}$ by a standard contradiction argument as in \cite[Theorem 9, Chap. 2]{protter1984}. Then \cite[Lemma 2]{chaabane1999} concludes positivity on $\Gamma_\beta$ using Hopf's lemma. The compactness argument of \cite[Lemma 3.2]{hu2015} is then adapted to our case to show  $u(x)\geq \eta$ for any $x\in \Gamma_{\beta,\epsilon}$.\\\\
(2) Corollary 1.1 in \cite[Chapter 8]{lions1973} shows that the solution $u$ to \eqref{eq:laplace} is analytic near and up to $\Gamma_\beta$. For $\delta$ small and $\tilde{U}:=\overline{\mathcal{O}}\cap((0,1)\times (-\delta,\delta))$ it further states that for $k=(k_1,k_2)$
$$\sup_{z\in \tilde{U}}|\partial^k u(z)|\leq C(M)(k!)C(M)^{|k|},$$
for $|k|\in \mathbb{N}_0$ and where $k!=k_1!k_2!$. Then the Taylor series of $u$ in $(\alpha,0)$ for any $\alpha \in [0,1]$ has a convergence radius of at least $r=C(M)^{-1}$. Indeed, for any $(x,y)$ with distance at most $r$ to $(\alpha,0)$ we have
\begin{align}
    u(z) = u(x,y) &= \sum_{n_1=0}^\infty \sum_{n_2=0}^\infty \frac{\partial^n u(\alpha,0)}{n!}(x-\alpha)^{n_1}y^{n_2},\\
    &\leq C(M) \sum_{n_1=0}^\infty \sum_{n_2=0}^\infty C(M)^{|n|}(x-\alpha)^{n_1}y^{n_2},\\
    &\leq \sum_{n_1=0}^\infty \sum_{n_2=0}^\infty (C(M)r)^{|n|} \label{eq:bound-on-fun}
\end{align}
where we denoted $n=(n_1,n_2)$. Since a power series is analytic in the interior of its region of convergence, $u$ is analytic in sufficiently small balls centered in $(\alpha,0)$. A covering argument then gives a unique analytic extension in for example $(0,1)\times (-\tilde{\delta},\tilde{\delta}))$ with $\tilde{\delta} = \min(\delta,(2C(M))^{-1})$. Repeating \eqref{eq:bound-on-fun} for $\partial^k u(z)$ for $k=1,2$, we note that 
\begin{equation}
    \|u\|_{C^2((0,1)\times (-\tilde{\delta},\tilde{\delta}))}\leq C(M).
\end{equation}
We also conclude $\Delta u= 0$ in $\mathcal{O} \cup ((0,1)\times (-\tilde{\delta},\tilde{\delta}))$. Indeed, $\Delta u$ is analytic in $\mathcal{O} \cup ((0,1)\times (-\tilde{\delta},\tilde{\delta}))$ and coincides with $0$ on $\mathcal{O}$ and hence is zero in $\mathcal{O} \cup ((0,1)\times (-\delta,\delta))$ by uniqueness of analytic functions.

\end{proof}
Using the general property of uniform analyticity up to the boundary we avoid the argument in \cite[Theorem 3.1]{hu2015}, which uses a reflection formula provided by \cite{belinskiy2008}. Inspection of this reflection formula reveals that we do need a condition like $\theta_i\in \mathcal{R}_2(M)$, $i=1,2,$ to reflect the solution to a possible small but fixed neighborhood. We can generalize our proof to stability estimates for any analytic $\Gamma_\beta$ for $d=2$ and $d=3$.

\section{Consistency for analytic functions}\label{sec:cons-analytic}
The result of Theorem \ref{thm:main-theorem} is derived from the property of posterior consistency, see \cite[Chapter 8]{ghosal2017} for a general treatment. In the following we address posterior consistency for analytic functions. We start by establishing a relationship between the space $\mathcal{A}_r(\Gamma_\beta)$ and the set of functions $\mathcal{R}_2(M)$. We prove a result, which is well-known and particularly simple in the setting of the $m$-dimensional $[-\pi,\pi)$-torus $\mathbb{T}^m$. Analogously, it generalizes to function spaces defined by the decay of the Fourier transform by the Paley-Wiener theorem, see \cite[Theorem IX.13]{reed1975}. We consider $m\geq 1$, since it follows in much the same way as $m=1$. To this end, let $\Gamma_\beta$ be an open compactly embedded subset of $[-\pi,\pi)^m$ for $m\in \mathbb{N}$. Let $\{\phi_k\}_{k\in \mathbb{Z}^m}$ be a real orthonormal Fourier basis and define for $f_k = \langle f, \phi_k\rangle_{L^2(\mathbb{T}^m)}$ the more general space
$$\mathcal{A}_{r,m}(\Gamma_\beta):=\{f=g|_{\Gamma_\beta}:g\in \mathcal{A}_r(\mathbb{T}^m)\}$$
with
$$\mathcal{A}_r(\mathbb{T}^m)=\{f\in L^2(\mathbb{T}^m):\|f\|^2_{r,\mathbb{T}^m} := \sum_{k\in \mathbb{Z}^m} |f_k|^2 e^{r|k|}<\infty\}$$
and the corresponding quotient norm of \eqref{eq:quotient-norm}, denoted $\|\cdot\|_{r,m}$. Note we write $|k|:=|k_1|+\hdots+|k_m|$ and not for example $\|k\|^2$ to make a sharper result. We keep the definition of $\mathcal{R}_2(M)$ as in Assumption \ref{assump:beta}, and note that the condition 
$$\sup_{x\in \overline{\Gamma}_\beta}|(\partial^k\beta)(x)|\leq M (k!)M^{|k|}$$
should be understood in multi-index notation (i.e. $\partial^k = \partial_{x_1}^{k_1}\hdots \partial_{x_m}^{k_m}$ and $k! = k_1!\hdots k_m!$) and for each $|k| \in \mathbb{N}_0$. Again this is closely related to the usual characterization of analytic functions on $\Gamma_\beta$, see \cite[Proposition 2.2.10]{krantz2002}. We also denote $d_\infty(x,S):=\inf_{y\in S} \|x-y\|_\infty$, the sup-norm distance of the point $x$ to the set $S$.
\begin{lemma}\label{lemma:fourier-analytic}
    Suppose $f\in \mathcal{A}_{r,m}(\Gamma_\beta)$, $r>0$ with $\|f\|_{r}\leq M_0$. Then,
    \begin{enumerate}[label=(\roman*)]
        \item there exists an analytic extension of $f$ to $G_r:=\{z\in \mathbb{C}: d_\infty(z,\overline{\Gamma}_\beta)\leq \frac{r}{4} \}$ with
        \begin{equation}
            \sup_{z\in G_r}|f(z)|\leq M_1
        \end{equation}
        for some $M_1=M_1(r,M_0,m)$.
        \item $\sup_{x\in\overline{\Gamma}_\beta}\left| (\partial^k) f(x) \right|\leq M_2 (k!)M_2^{|k|}$
        for some $M_2=M_2(M_1,r)$.
        \item $f\in \mathcal{R}_2(M)$ for some $M=M(M_1)$.
    \end{enumerate}
\end{lemma}
\begin{proof}
    We complete the proof for $m=2$ and note that the case for other $m\in \mathbb{N}$ follows in the same way.\\\\
    $(i)$ By assumption $f$ is the restriction of a function in $\mathcal{A}_r(\mathbb{T}^2)$ with $\|f\|_{r,\mathbb{T}^2}\leq M_0$, which implies for the usual Fourier coefficients $\hat{f}_k := \frac{1}{2\pi}\langle f, e^{ik\cdot x}\rangle_{L^2(\mathbb{T}^2)}$
    $$|\hat{f}_k|\leq M_0 e^{-\frac{r}{2}|k|}.$$
    Define the `polycylinder'
    $$P_\rho = \{w\in \mathbb{C}^2: |w_1|<e^{\rho/2}, |w_2|<e^{\rho/2}\}.$$
    Take a compact set $K\subset P_r$, then for any $w\in K$, the family of functions $\{\hat{f}_k w^k\}_{k\in \mathbb{N}_0^2}$ (where $w^k=w_1^{k_1}w_2^{k_2}$)
    is bounded. Then by the argument of \cite[Corollary 1.5.9.2]{scheidemann2005}, the function $w\mapsto \sum_{k \in \mathbb{N}^2_0} \hat{f}_kw^k$ is complex analytic in $P_r$. In fact, by the same argument the four power series 
    \begin{equation}\label{eq:power-series}
       w\mapsto \sum_{k\in \mathbb{N}_0^2} \hat{f}_{\pm k_1, \pm k_2} w^k 
    \end{equation}
    are complex analytic in $P_r$. Further, for all $w\in \overline{P}_{r/2}$
    \begin{equation}\label{eq:upper-bound-on-series}
        \left | \sum_{k\in \mathbb{N}_0^2} \hat{f}_{\pm k_1, \pm k_2} w^k \right | \leq \sum_{k\in \mathbb{N}_0^2} |\hat{f}_{\pm k_1, \pm k_2}||w|^k \leq C(M_0,r) 
    \end{equation}
    Now decompose the following Laurent series into four similar power series as
    \begin{align}
        \sum_{k\in \mathbb{Z}^2} \hat{f}_k w^k  &= \sum_{k_1 = 1}^\infty \sum_{k_2 = 0}^\infty \hat{f}_{-k_1,k_2} w_1^{-k_1} w_2^{k_2},\\
        &\,\,\,\,+ \sum_{k_1 = 1}^\infty \sum_{k_2 = 1}^\infty \hat{f}_{-k_1,-k_2} w_1^{-k_1} w_2^{-k_2},\\
        &\,\,\,\,+ \sum_{k_1 = 0}^\infty \sum_{k_2 = 0}^\infty \hat{f}_{k_1,k_2} w_1^{k_1} w_2^{k_2},\\
        &\,\,\,\,+ \sum_{k_1 = 0}^\infty \sum_{k_2 = 1}^\infty \hat{f}_{k_1,-k_2} w_1^{k_1} w_2^{-k_2}.
    \end{align}
Consider first the first term. This has the form $w\mapsto g(w_1^{-1},w_2)$ for a function $g_1$ on the form \eqref{eq:power-series} complex analytic in $P_\rho$. The function $w\mapsto (w_1^{-1},w_2)$ is complex analytic in for example $\{w\in\mathbb{C}^2:w_1 > e^{-r/2}, w_2 < e^{r/2}\}$, since $w\mapsto w_i$ is complex analytic everywhere and $w\mapsto w_i^{-1}$ is complex analytic for $w_i$ away from zero, $i=1,2$, see \cite[Proposition 1.2.2]{scheidemann2005}. Then also $w\mapsto g_1(w_1^{-1},w_2)$ is complex analytic in $\{w\in\mathbb{C}^2:w_1 > e^{-r/2}, w_2 < e^{r/2}\}$, since compositions of analytic functions are analytic, see again \cite[Proposition 1.2.2]{scheidemann2005}. Continuing this argument for each term above, we find that 
$$g(w):=\sum_{k\in\mathbb{Z}^2} \hat{f}_k w^k$$
is complex analytic in the `polyannulus' $\{w\in \mathbb{C}^2: e^{-r/2}<w_i < e^{r/2},i=1,2\}$. Using that $z\mapsto e^z$ is entire on $\mathbb{C}$ and again the composition rule, we find that
$$z\mapsto g(e^{iz_1},e^{iz_2}) = \sum_{k\in\mathbb{Z}^2} \hat{f}_k e^{ik\cdot z}$$
is complex analytic in $\{z\in \mathbb{C}^2: |\mathrm{Im}(z_i)|<r/2, i=1,2\}$. Moreover, since \eqref{eq:upper-bound-on-series} is a bound for each of the four power series which make up $f$ and $G_r$ is a subset of the strip of where it is defined, we conclude
\begin{equation}\label{eq:bound-in-paley}
    \sup_{z\in G_r}|f(z)|\leq M_1(r,M_0).
\end{equation}
$(ii)$ The Cauchy integral inequality in \cite[Theorem 1.3.3]{scheidemann2005} gives the estimate 
$$\sup_{|z_i| < r/4, i=1,2} |(\partial^k f)(z)|\leq (k!) (r/4)^{|k|} \sup_{|z_i| = r/4, i=1,2}|f(z)|.$$
Since $\overline{\Gamma}_\beta$ is compact, we can cover it by real translations of $\{z\in \mathbb{C}^2: |z_i| < r/4, i=1,2\}$ and conclude by \eqref{eq:bound-in-paley} that there exists a constant $M_2=M_2(M_1,r)$ such that
$$\sup_{x\in \overline{\Gamma}_\beta} |(\partial^k f)(x)|\leq M_2 (k!)M_2^{|k|}.$$
$(iii)$ Since $z\mapsto e^z$ is entire, also $z\mapsto e^{f(z)}$ is complex analytic in $G_r$ with a bound $\sup_{z\in G_r}|e^{f(z)}|\leq e^{M_1}$. Repeating the same arguments as of $(ii)$ we conclude that $f\in \mathcal{R}_2(M)$ for some $M=M(M_1)$.
\end{proof}

We now return to the question of consistency, which involves precise statements on the prior we use. Since $\tilde{\Pi}_2$ is a Gaussian measure in $C(\overline{\Gamma}_\beta)$, a covering number bound of the unit norm-ball in the RKHS $\mathcal{H}_2=\mathcal{A}_{r}(\Gamma_\beta)$, yields a bound on the measure of small norm balls, see \cite{li1999}. To make use of this, we define the notion of covering numbers as follows. Let the covering number $N(A,d,\rho)$ for $A\subset X$ of some space $X$ endowed with a semimetric $d$, denote the minimum number of closed $d$-balls $\{x\in X: d(x_0,x)\leq \rho\}$ with center $x_0\in A$ and radius $\rho>0$ needed to cover $A$, see for example \cite[Appendix C]{ghosal2017} or \cite[Section 4.3.7]{gine2016}. When $d$ is replaced by a norm, we mean the metric induced by the norm. The following consequence of \cite[Proposition C.9]{ghosal2017} allows us to bound the unit norm ball of $\mathcal{A}_r(\Gamma_\beta)$.
\begin{lemma}\label{lemma:covering-number-analytic}
    The class $A(M_1)$ of all functions $f:[0,1]^m\rightarrow \mathbb{R}$ that can be extended to a complex analytic function on $G_r$ with $\sup_{z\in G_r}|f(z)|\leq M_1$ for some $M_1>0$ and $r>0$, satisfies for all $\rho>0$ sufficiently small
    \begin{equation}\label{eq:wanted-log-N}
        \log N(A(M_1), \|\cdot\|_{\infty},\rho) \leq C(r,m,M_1)\log(\rho^{-1})^{1+m}.
    \end{equation}
\end{lemma}
\begin{proof}
    Proposition C.9 in \cite{ghosal2017} states that
    $$\log N(A(1), \|\cdot\|_{\infty},\rho)\leq C(m)r^{-m}\log(\rho^{-1})^{1+m}$$
    for all $0<\rho<1/2$. Since $\|\cdot\|_{\infty}\leq \|\cdot\|_{C([0,1]^m)}$ \cite[eq. (4.172)]{gine2016} gives
    $$\log N(A(1), \|\cdot\|_{\infty},\rho)\leq \log N(A(1), \|\cdot\|_{C([0,1]^m)},\rho).$$
    Then, combining the two last displays with \cite[eq. (4.171)]{gine2016} we have 
    \begin{align}
       \log N(A(M_1), \|\cdot\|_{\infty},\rho) &= \log N(A(1), \|\cdot\|_{\infty},\rho M_1^{-1}),\\
       &\leq \log N(A(1), \|\cdot\|_{C([0,1]^m)},\rho M_1^{-1}),\\
       &\leq C(r,m)\log(M_1 \rho^{-1})^{1+m}.
    \end{align}
    By the convexity of $x\mapsto x^{1+m}$, $m\geq 1$, we have the inequality $(x+y)^{m+1}\leq 2^{m}(x^{m+1}+y^{m+1})$ for $x,y\in \mathbb{R}$. Then for $\rho$ small enough, 
    $$\log(M_1 \rho^{-1})^{1+m} \leq C(M_1,m)\log(\rho)^{1+m},$$
    and hence \eqref{eq:wanted-log-N} is satisfied. 
\end{proof}

Since we constructed $\Pi_2$ for $m=1$, we consider from now only this case, although everything generalizes to higher dimensions. See also Remark \ref{remark:generalization-dimension} below.
To this end, denote the unit norm ball of $\mathcal{H}_2=\mathcal{A}_r(\Gamma_\beta)$ by
$$B_{\mathcal{H}_2}:=\{f\in \mathcal{H}_2:\|f\|_{\mathcal{H}_2}\leq 1\}.$$
Note that $B_{\mathcal{H}_2}\subset A(M_1)$ for some $M_1=M_1(r)$ by Lemma \ref{lemma:fourier-analytic} (i).

\begin{lemma} \label{lemma:small-ball-lemma}
Let $\phi(\rho):=-\log \tilde{\Pi}_2(\theta \in C(\overline{\Gamma}_\beta):\|\theta\|_{\infty} \leq \rho)$
where $\tilde{\Pi}_2$ is dependent on $r>0$. For all $\rho>0$ sufficiently small,
    \begin{equation}\label{eq:small-ball-prob-bound}
    \phi(\rho) \leq C(r) \log(\rho^{-1})^{2}.
\end{equation}
\end{lemma}
\begin{proof}
    We follow \cite[Lemma 4.6]{vaart2009}. Theorem 1.2 of \cite{li1999} initially gives the estimate
    \begin{equation}\label{eq:crude-bound}
        \phi(\rho)\leq C(r,M_1)\rho^{-2},
    \end{equation}
for all $\rho>0$ sufficiently small, since 
$$\log N(B_{\mathcal{H}_2}, \|\cdot\|_{\infty},\rho) \leq C(r) \rho^{-1},$$
for all $\rho>0$ sufficiently small by Lemma \ref{lemma:covering-number-analytic} for $m=1$. The first display of the proof of Lemma 4.6 in \cite{vaart2009} provides the inequality
$$\phi(2\rho)\leq \log N(B_{\mathcal{H}_2},\|\cdot\|_{\infty},2\rho [2\phi(\rho)]^{-1/2}).$$
Combining this with \eqref{eq:crude-bound} then gives \eqref{eq:small-ball-prob-bound}.
\end{proof}

The following result corresponds to Theorem 2.2.2 of \cite{nickl2023} for rescaled Gaussian priors for analytic functions. We define $d_{\mathcal{G}}(\theta_1,\theta_2):=\|\mathcal{G}(\theta_1)-\mathcal{G}(\theta_2)\|_{L^2(\Gamma)}$ for all $\theta_1,\theta_2\in \Theta$.

\begin{lemma}\label{lemma:analytic-consistency}
    Let $\theta_0 \in \mathcal{H}_2=\mathcal{A}_r(\Gamma_\beta)$, $r>0$,
    and $\Pi_2$ be as defined in \eqref{eq:prior-def}. Set,
    \begin{equation}\label{eq:choice-of-delta}
        \delta_N = N^{-1/2}\log(N).
    \end{equation}
    Let $U>0$ be large enough depending on $\tilde{\Pi}_{2}$, $r$ and such that $\|\mathcal{G}(\theta_0)\|_{C(\overline{\Gamma})}\leq U$.
    Then, there exists Borel measurable sets $\Theta_N$ such that
    \begin{enumerate}[label=(\roman*)]
        \item $\Pi(\theta:d_{\mathcal{G}}(\theta,\theta_0)\leq \delta_N, \|\mathcal{G}(\theta)\|_{C(\overline{\Gamma})}\leq U) \geq e^{-C_1 N\delta_N^2}$ for some $C_1>0$,
        \item $\Pi_2(\Theta_N^c) \leq e^{-C_2N\delta_N^2}$ for $C_2>C_1+2$.
        \item $\log N(\Theta_N, d_{\mathcal{G}}, m_0\delta_N) \leq C(C_2,r)N\delta_N^2$ for $m_0>0$ large enough 
    \end{enumerate}
    for all $N$ sufficiently large. 
\end{lemma}
\begin{proof}
First we give the form of $\Theta_N$. Define $B_{\mathcal{H}_2}(\delta)$ and  $B_{\infty}(\delta)$ to be the closed norm balls of radius $\delta>0$ in $\mathcal{H}_2$ and $C(\overline{\Gamma}_\beta)$, respectively. That is,
\begin{align}
    B_{\mathcal{H}_2}(\delta)&:=\{f\in \mathcal{H}_2: \|f\|_{\mathcal{H}_2}\leq \delta\},\\
    B_{\infty}(\delta)&:=\{f\in C(\overline{\Gamma}_\beta): \|f\|_{\infty}\leq \delta\},
\end{align}
Recall, also that $B_{\mathcal{H}_2}(M_0)\subset A(M_1)$ for some $M_1=M_1(r,M_0)$ by Lemma \ref{lemma:fourier-analytic} (i).
Then we take
\begin{equation}\label{eq:theta-N-def}
    \Theta_N := (B_{\mathcal{H}_2}(M) + B_{\infty}(M\delta_N)) \cap \mathcal{R}_2(M).
\end{equation}
for $M>0$ sufficiently large determined by $(ii)$ below.\\\\
We also recall the following triangle inequality fact needed for $(ii)$ below: a $C\delta_N$-covering of $B_{\mathcal{H}_2}(M)$ is a $(M+C)\delta_N$-covering of $B_{\mathcal{H}_2}(M) + B_\infty(M\delta_N)$ so that
\begin{equation}
    N(B_{\mathcal{H}_2}(M) + B_\infty(M\delta_N),\|\cdot\|_{\infty},(M+C)\delta_N) \leq N(B_{\mathcal{H}_2}(M),\|\cdot\|_{\infty},C\delta_N).
\end{equation}
This implies for $\tilde{C}$ large enough that
\begin{equation}\label{eq:triangle-property}
    N(\Theta_N,\|\cdot\|_{\infty},\tilde{C}\delta_N) \leq N(B_{\mathcal{H}_2}(M),\|\cdot\|_{\infty},(\tilde{C}-M)\delta_N).
\end{equation}
In addition, we will use repeatedly below that
$$\kappa_{N,2}=\frac{1}{\sqrt{N}\delta_N}.$$
    \noindent $(i)$ We proceed as in \cite[Theorem 2.2.2]{nickl2023}. Recall that $\tilde{\Pi}_2(\mathcal{A}_q(\Gamma_\beta))=1$ for any $0<q<r$. Hence also $\Pi_2(\mathcal{A}_q(\Gamma_\beta))=1$. Fernique's theorem \cite[Theorem 2.1.20]{gine2016} initially gives that $\mathbb{E}[\|\tilde{\theta}_2\|_{q}]\leq D$ for some constant $D$ depending only on the prior $\tilde{\Pi}_{2}$, and next
    \begin{align}\label{eq:excess-mass1}
       \Pi_2(\theta:\|\theta\|_q>M_0) &= \tilde{\Pi}_2(\tilde{\theta}:\|\tilde{\theta}\|_q>M_0\sqrt{N}\delta_N),\\
       &\leq \tilde{\Pi}_{2}(\tilde{\theta}:\|\tilde{\theta}\|_q-\mathbb{E}[\|\tilde{\theta}\|_q]>\frac{1}{2}M_0\sqrt{N}\delta_N),\\
       &\leq e^{-CM_0^2N\delta_N^2},
    \end{align}
    for some sufficiently large constant $M_0=M_0(D)$ and some constant $C=C(\tilde{\Pi}_2)$. By Lemma \ref{lemma:fourier-analytic}, this implies
     \begin{equation}\label{eq:half-prob}
    \Pi_2(\theta:\|\theta\|_{H^1(\Gamma_\beta)}> M_1)\leq e^{-CM_0^2N\delta_N^2}\leq \frac{1}{2},
    \end{equation}
    for $M_0$ large enough depending on $C$ and $M_1=M_1(M_0,r)$.
   Note we have
    $$\|\theta-\theta_0\|_{H^1(\Gamma_\beta)}\leq M_1 \quad \Rightarrow \quad \|\theta\|_{H^1(\Gamma_\beta)} \leq M_1+\|\theta_0\|_{H^1(\Gamma_\beta)} \equiv \bar{M},$$
    which by Lemma \ref{lemma:forward-reg} implies $\|\mathcal{G}(\theta)\|_{C(\overline{\Gamma})}\leq U=U(\bar{M})$, since $$\|e^\theta\|_{H^1(\Gamma_\beta)}\lesssim \|e^\theta\|_{\infty}(1+\|\theta\|_{H^1(\Gamma_\beta)}).$$ 
    Using again Lemma \ref{lemma:forward-reg} and Corollary 2.6.18 \cite{gine2016} permitted since $\theta_0\in \mathcal{H}_2$ and $\Pi_2(\Theta)=1$, we get
    \begin{align}
        \Pi_2(d_{\mathcal{G}}(\theta,\theta_0)&\leq \delta_N, \|\mathcal{G}(\theta)\|_{C(\overline{\Gamma})}\leq U)\\
        &\geq \Pi_2(d_{\mathcal{G}}(\theta,\theta_0)\leq \delta_N, \|\theta-\theta_0\|_{H^1(\Gamma_\beta)}\leq M_1),\\
        &\geq \Pi_2(\|\theta-\theta_0\|_{\infty}\leq K^{-1}\delta_N, \|\theta-\theta_0\|_{H^1(\Gamma_\beta)}\leq M_1),\\
        &\geq e^{-\frac{1}{2}\|\theta_0\|_{\mathcal{H}_{2,N}}^2} \Pi_2(\|\theta\|_{\infty}\leq K^{-1}\delta_N, \|\theta\|_{H^1(\Gamma_\beta)}\leq M_1),\\
        &\geq  e^{-C(\theta_0)\kappa_{N,2}^{-2}} \Pi_2(\|\theta\|_{\infty}\leq K^{-1}\delta_N) \Pi_2(\|\theta\|_{H^1(\Gamma_\beta)}\leq M_1),
    \end{align}
    where we used the Gaussian correlation inequality, see \cite[Theorem 6.2.2]{nickl2023}, and the relation $\|\theta_0\|_{\mathcal{H}_{2,N}}=\kappa_{N,2}^{-1}\|\theta_0\|_{\mathcal{H}_2}$ for the last line. We also note that $K$ depends on $M_1$. Lemma \ref{lemma:small-ball-lemma} implies
    \begin{align}
        -\log \Pi_2(\theta:\|\theta\|_{\infty}\leq K^{-1}\delta_N) &= -\log \tilde{\Pi}_{2}(\tilde{\theta}:\|\tilde{\theta}\|_{L^\infty(\Gamma_\beta)}\leq \kappa_{N,2}^{-1}K^{-1}\delta_N)\\ 
        &\leq C(r) \log \left( \frac{K\kappa_{N,2}}{\delta_N} \right)^2,\\ 
        &= C(r) \log \left( K \sqrt{N} \log(N)^{-2} \right)^2,\\
        &\leq C(r) \left[ \log(K)+\frac{1}{2}\log(N)-2\log(\log(N)) \right]^2,\\
        &\leq C(K,r) \log(N)^2,\\
        &\leq C(K,r)N \delta_N^2, \label{eq:small-ball-est}
    \end{align}
    for a sufficiently large constant $C=C(K,r)$. Equation \eqref{eq:half-prob} shows
    $$\Pi_2(\theta:\|\theta\|_{H^1(\Gamma_\beta)}\leq M_1)\geq 1- \frac{1}{2}= \frac{1}{2}.$$
    The three last displays shows $(i)$ for \eqref{eq:choice-of-delta} and a constant $C_1=C_1(\theta_0,K,M_1,r)$.\\
    $(ii)$ Lemma \ref{lemma:fourier-analytic} implies there exists $M=M(q,M_0)$ such that 
    $$\{f\in \mathcal{A}_q(\Gamma_\beta): \|f\|_q\leq M_0\}\subset \mathcal{R}_2(M),$$
    and hence by \eqref{eq:excess-mass1} we can pick $M_0$ large enough dependent on $C_2$ such that
    \begin{equation}
        \Pi_2(\mathcal{R}_2(M)^c)\leq \frac{1}{2}e^{-C_2N\delta_N^2}.
    \end{equation}
    We simply pick $q=r/2$ to fix constants. Then it suffices to prove
    \begin{equation}
        \Pi_2(B_{\mathcal{H}_2}(M) + B_{\infty}(M\delta_N)) \geq 1-\frac{1}{2}e^{-C_2N\delta_N^2}.
    \end{equation}
    We prove the stronger bound 
    \begin{equation}
        \Pi_2(B_{\mathcal{H}_2}(M) + B_{\infty}(M\delta_N)) \geq 1-e^{-2C_2N\delta_N^2}.
    \end{equation}
    By similar computations as with \eqref{eq:small-ball-est} for $M\geq 1$, we find
    \begin{align}
        -\log \Pi_2(\theta:\|\theta\|_{\infty}\leq M\delta_N) &\leq C(r) \log(\delta_N^{-1}),\\
        &\leq C(r)(1/2\log(N)-\log\log(N)),\\
        &\leq 2C_2\log(N)^2,\\
        &\leq 2C_2N\delta_N^2,
    \end{align}
    for any given $C_2>0$ and $N$ sufficiently large.
    As in \cite[Theorem 2.2.2]{nickl2023} we denote
    $$B_N=-2\Phi^{-1}(e^{-2C_2N\delta_N^2}),$$
    where $\Phi$ is the standard normal cumulative distribution function. Then by \cite[Lemma K.6]{ghosal2017} we have
    $$B_N\leq 2\sqrt{2\log(e^{2C_2N\delta_N^2})}\leq 4 \sqrt{C_2}\sqrt{N}\delta_N.$$
    Then for $M>4\sqrt{C_2}$ such that $B_N\leq M\sqrt{N}\delta_N$ we use the isoperimetric inequality \cite[Theorem 2.6.12]{gine2016} to conclude that
    \begin{align}
        \Pi_2(B_{\mathcal{H}_2}(M) + B_\infty(M\delta_N)) &= \tilde{\Pi}_2(B_{\mathcal{H}_2}(M\sqrt{N}\delta_N) + B_\infty(M\sqrt{N}\delta_N^2)),\\
        &\geq \tilde{\Pi}_2(B_{\mathcal{H}_2}(B_N) + B_\infty(M\sqrt{N}\delta_N^2)),\\
        &\geq \Phi(\Phi^{-1}[\tilde{\Pi}_2(B_\infty(M\sqrt{N}\delta_N^2))] +B_N),\\
        &\geq \Phi(\Phi^{-1}[e^{-2C_2N\delta_N^2}]+B_N),\\
        &=\Phi(-\Phi^{-1}[e^{-2C_2N\delta_N^2}]),\\
        &= 1-\Phi(\Phi^{-1}[e^{-2C_2N\delta_N^2}]),\\
        &= 1- e^{-2C_2N\delta_N^2},
    \end{align}
    using also $\Phi(-x)=1-\Phi(x)$.\\
     $(iii)$ We recall that $B_{\mathcal{H}_2}(M)\subset A(M_1)$ for some $M_1=M_1(M,r)$ by Lemma \ref{lemma:fourier-analytic} so that Lemma \ref{lemma:covering-number-analytic} gives
    \begin{align}
        \log N(B_{\mathcal{H}_2}(M),\|\cdot\|_{\infty},\delta_N)&\leq C(r,M)\log(\delta_N^{-1})^2,\\
        &\leq C(r,M)\left(1/2\log(N)-\log\log(N)\right)^2,\\
        &\leq C(r,M) N\delta_N^2,
    \end{align}
    for $N$ large enough.
    Then using Lemma \ref{lemma:forward-reg} with $m_0=m_0(K,M)$ sufficiently large and \eqref{eq:triangle-property} we get
    \begin{align}
         \log N(\Theta_N, d_{\mathcal{G}}, m_0\delta_N) &\leq \log N(\Theta_N, \|\cdot\|_{\infty}, K^{-1}m_0\delta_N),\\
         &\leq \log N(B_{\mathcal{H}_2}(M), \|\cdot\|_{\infty}, (K^{-1}m_0-M)\delta_N),\\
         &\leq C(r,M) N\delta_N^2,
    \end{align}
    Note $M$ depends only on $r$ and $C_2$ through $M_0$.
\end{proof}
\begin{remark}\label{remark:generalization-dimension}
Extending Lemma $\ref{lemma:small-ball-lemma}$ and $\ref{lemma:analytic-consistency}$ to $m>1$ and other exponential decay is straightforward. Indeed, define a Gaussian prior by the restriction to $\Gamma_\beta \subset [-\pi,\pi)^m$ of the random series
\begin{equation}
    \tilde{\theta}_2 = \sum_{k \in \mathbb{Z}^m} g_k e^{-\frac{r}{2} |k|}\phi_k, \qquad g_k \stackrel{i.i.d}{\sim} N(0,1).
\end{equation}
    This is an element of $\mathcal{A}_{q,m}(\Gamma_\beta)$ a.s for $q<r$ and its RKHS is $\mathcal{H}_2=\mathcal{A}_{r,m}(\Gamma_\beta)$. Then Lemma \ref{lemma:small-ball-lemma} follows in the same way by noting $B_{\mathcal{H}_2}\subset A(M_1)$ for some $M_1=M_1(r)$ by Lemma \ref{lemma:fourier-analytic} (i). Given a Lipschitz continuous forward map $\mathcal{G}$, Lemma \ref{lemma:analytic-consistency} follows for $\delta_N=N^{-1/2}\log(N)^\zeta$ for some exponent $\zeta$ dependent on $m$.
\end{remark}
\begin{proof}[Proof of \ref{thm:main-theorem} (ii)]
    By Lemma \ref{lemma:analytic-consistency}, conditions $(1.32)$ and $(1.33)$ of Theorem 1.3.2 \cite{nickl2023} are satisfied for the choice \eqref{eq:theta-N-def} of $\Theta_N$.  Lemma \ref{lemma:analytic-consistency} $(ii)$ and the bound on the Hellinger distance $h(p_\theta,p_\vartheta)\leq \frac{1}{2}d_{\mathcal{G}}(\theta,\vartheta)$, see \cite[Proposition 1.3.1]{nickl2023}, implies 
    \begin{equation}\label{eq:bound-in-last-theorem}
        N(\tilde{\Theta}_N,h,\frac{1}{2}m_0\delta_N)\leq N(\Theta_N,d_{\mathcal{G}},m_0\delta_N)\leq e^{C(C_2,r)N\delta_N^2}
    \end{equation}
    hence for all $\varepsilon>2m_0\delta_N$
    \[N(\tilde{\Theta}_N,h,\frac{\varepsilon}{4})\leq e^{C(C_2,r)N\delta_N^2},\]
    with
    \[\tilde{\Theta}_N:=\{p_\theta:\theta\in \Theta_N\}.\]
    Note the right-hand side of \eqref{eq:bound-in-last-theorem} is independent of such $\varepsilon$. Setting $\varepsilon=m\delta_N$ for $m>2m_0$,  Theorem 7.1.4 of \cite{gine2016} gives the existence of statistical tests $\Psi_N:(\mathbb{R}\times \Gamma)^N\rightarrow \{0,1\}$ satisfying
    \[P_{\theta_0}^N(\Psi_N =1 ) \rightarrow 0,\]
    as $N\rightarrow \infty$, and
    \[\sup_{\theta\in \Theta_N:h(p_\theta,p_{\theta_0})>m\delta_N} E_\theta^N(1-\Psi_N)\leq e^{-\kappa N\delta_N^2}\]
    for $m$ large enough also depending on $C(C_2,r)$ and $\kappa$. Then the proof of Theorem 1.3.2 \cite{nickl2023} implies that for all $0<b<C_2-C_1-2$ we can choose $C_0=C_0(C_1,C_2,r,m_0,b,U)$ large enough such that
    \[P_{\theta_0}^N \left( \Pi_N(\theta \in \Theta_N:d_{\mathcal{G}}(\theta,\theta_0)\leq C_0 \delta_N|D_N)\leq 1-e^{-bN\delta_N^2} \right)\rightarrow 0.\]
    Lemma \ref{lemma:stability} (ii) implies
    \[\{\theta\in \Theta_N: d_{\mathcal{G}}(\theta,\theta_0)\leq C_0\delta_N \} \subset \{\theta \in \Theta_N: \|\theta-\theta_0\|_{L^2(\Gamma_{\beta,\epsilon})}\leq KC_0^\sigma \delta_N^\sigma\}\]
    so that we also have
    \[P_{\theta_0}^N \left( \Pi_N(\theta \in \Theta_N: \|\theta-\theta_0\|_{L^2(\Gamma_{\beta,\epsilon})}\leq KC_0^\sigma \delta_N^\sigma|D_N)\leq 1-e^{-bN\delta_N^2} \right)\rightarrow 0.\]
    Then the argument of Theorem 2.3.2 \cite{nickl2023} applies in the same way here to the effect that
    \[\|\mathbb{E}_2[\theta | D_N] - \theta_0\|_{L^2(\Gamma_{\beta,\epsilon})} \rightarrow 0 \qquad \text{in $P_{\theta_0}^N$-probability}\]
    with rate $\delta_N^\sigma$ as $N\rightarrow \infty$.
\end{proof}

\section{Additional figures from experiments}\label{sec:morefigures}

\begin{figure}[H]
\centering
  \includegraphics[width=0.95\linewidth]{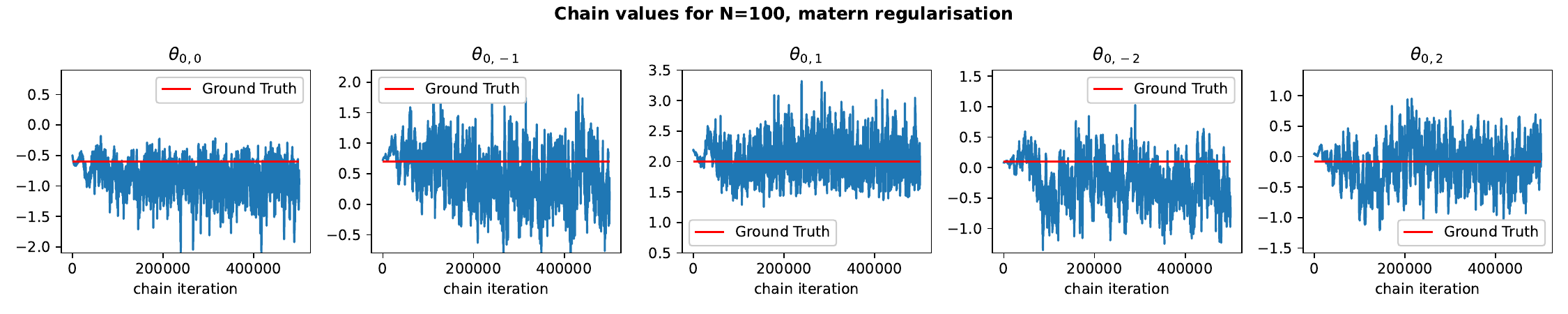}

  \centering
  \includegraphics[width=0.95\linewidth]{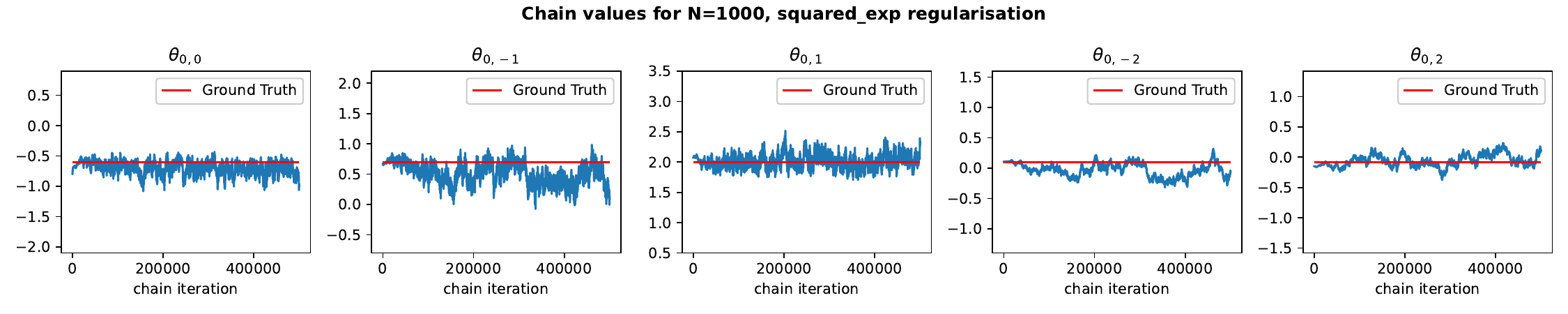}

\caption{Chains from the MCMC from the Laplace problem with respectively Matérn regularisation and $N=100$ and squared exponential regularisation and $N=1000$. The method used is adaptive Monte Carlo as described in \ref{sec:mcmc}, with a total of 500k iterations.}
\label{fig:chains}
\end{figure}

\begin{figure}[H]
\centering
  \includegraphics[width=0.95\linewidth]{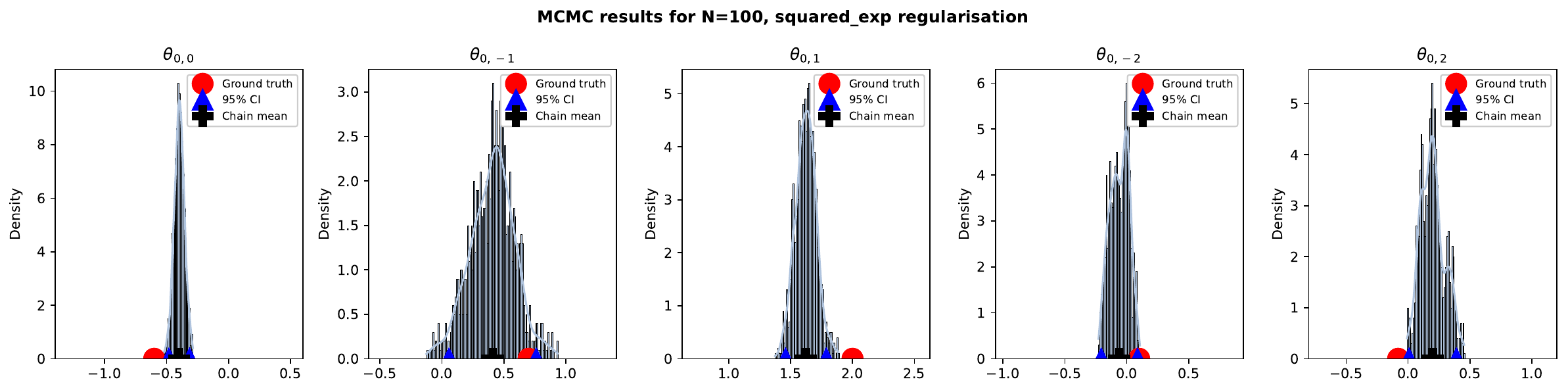}

  \centering
  \includegraphics[width=0.95\linewidth]{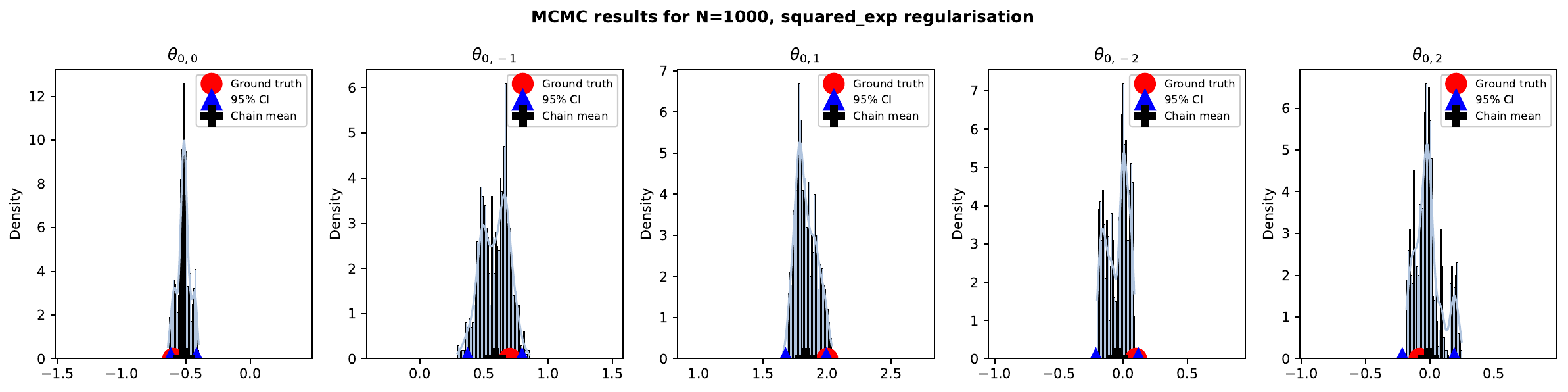}

\caption{Posterior densities from the Stokes experiments \ref{sec:stokes}, constructed using 1000 samples taken equidistantly from the chain after removal of the burn-in.
}
\label{fig:stokes_posterior}
\end{figure}

\bibliographystyle{siamplain}
\bibliography{references}

\end{document}